\documentclass[11pt,a4paper]{article}

\usepackage{amssymb,amsthm,amsmath,amsfonts}
\usepackage[pdftex]{graphicx}
\usepackage{graphicx}
\usepackage{bm}
\usepackage[margin=1in]{geometry}
\usepackage{color}
\usepackage[toc,page]{appendix}
\usepackage{hyperref}
\usepackage[export]{adjustbox}
\usepackage{caption}
\usepackage{array}
\usepackage[skip=1ex, belowskip=2ex]{subcaption}

\usepackage{tikz,amsmath}
\usetikzlibrary{shapes.geometric, arrows.meta,shadows,positioning}
\tikzstyle{process} = [rectangle,copy shadow={fill=black,shadow xshift=-0.5ex,shadow yshift=-0.5ex}, minimum width=3em, minimum height=2em, text centered, draw=black, fill=gray!10,]

\tikzstyle{arrow} = [thick,->,>=stealth]
\numberwithin{equation}{section}
\newtheorem{theorem}{Theorem}[section]
\newtheorem{lemma}[theorem]{Lemma}
\newtheorem{proposition}[theorem]{Proposition}

\theoremstyle{definition}

\theoremstyle{remark}

\newtheorem{assume}{Assumption}

\usepackage{hyperref}
\usepackage{comment}

\title{Extension on the convergence rates of moment-SOS hierarchies via approximation of truncated moment sequences}
\author{Hoang Anh Tran\thanks{Department of Mathematics, National University of Singapore, Singapore 119076. (tranhoanganh@u.nus.edu).},~ 
Kim-Chuan Toh\thanks{Department of Mathematics, and Institute of Operations Research and Analytics, National University of Singapore, Singapore 119076. (mattohkc@nus.edu.sg).}}

\newcommand{\bx}{\mathbf{x}}
\newcommand{\bu}{\mathbf{u}}
\newcommand{\bt}{\mathbf{t}}
\newcommand{\biw}{\bm{\mathit{w}}}
\newcommand{\cD}{\mathcal{D}}
\newcommand{\biy}{\bm{\mathit{y}}}
\newcommand{\biz}{\bm{\mathit{z}}}
\newcommand{\by}{\mathbf{y}}
\newcommand{\bz}{\mathbf{z}}
\newcommand{\RR}{\mathbb{R}}
\newcommand{\NN}{\mathbb{N}}
\newcommand{\bv}{\mathbf{v}}
\newcommand{\CP}{\mathcal{P}}
\newcommand{\CQ}{\mathcal{Q}}
\newcommand{\CX}{\mathcal{X}}
\newcommand{\cX}{\mathcal{X}}
\newcommand{\CT}{\mathcal{T}}
\newcommand{\bM}{\mathbf{M}}

\newcommand{\fmin}{f_{\min}}

\newcommand{\mlb}{\mathrm{mlb}}
\newcommand{\lb}{\mathrm{lb}}

\newcommand{\CM}{\mathcal{M}}

\newcommand{\bd}{\mathbf{d}}

\newcommand{\CZ}{\mathcal{Z}}
\newcommand{\CY}{\mathcal{Y}}

\newcommand{\cS}{\mathcal{S}}
\newcommand{\CR}{\mathcal{R}}
\newcommand{\bS}{\mathbb{S}}
\newcommand{\cQ}{\mathcal{Q}}

\newcommand{\bR}{\mathbb{R}}
\newcommand{\bg}{\mathbf{g}}
\newcommand{\Lo}{\text{\L{}}}
\newcommand{\iy}{\mathit{y}}
\newcommand{\vc}{\mathbf{vec}}
\newcommand{\bp}{\mathbf{p}}

\newcommand{\Lip}{\textrm{Lip}}
\newcommand{\cK}{\mathcal{K}}
\newcommand{\bh}{\mathbf{h}}
\newcommand{\tr}{\operatorname{tr}}

\def\degg#1{\lceil#1\rceil}

\def\pMTX2r{\mathcal{P\!M}(\CT(\CX)_{2r})}
\def\pMQX2r{\mathcal{P\!M}(\CQ(\CX)_{2r})}
\def\pMRX2r{\mathcal{P\!M}(\CR(\CX)_{r})}
\def\pMkTX2r{\mathcal{P\!M}_k(\CT(\CX)_{2r})}
\def\pMkQX2r{\mathcal{P\!M}_k(\CQ(\CX)_{2r})}
\def\pMkRX2r{\mathcal{P\!M}_k(\CR(\CX)_{r})}

\newcommand{\interior}[1]{%
  {\kern0pt#1}^{\mathrm{o}}%
}
\allowdisplaybreaks
\hypersetup{hypertexnames=false}
    
\def\cheb#1{\|#1\|_{1,\mathrm{Cheb}}}
\def\mon#1{\|#1\|_{1,\mathrm{mon}}}
\def\monc#1{\|#1\|_{2,\mathrm{mon}}}

\newcommand{\supnorm}[2]{\|#1\|_{\infty,#2}}

\newcommand{\op}{\mathrm{op}}

\begin{document}
\maketitle

\begin{abstract}
This paper continues our work on the convergence rates of moment-SOS hierarchies via approximation of truncated moment sequences. We extend the method !\cite{tran2025convergence} developed for the Schm\"udgen-type hierarchy to the Putinar-type, Krivine--Stengle-type, extended-Handelman-type, and Hol-Scherer-type hierarchies. The main idea is to lift a pseudo-moment sequence to a simple set, approximate it by a moment sequence, and project the atoms of a representing measure onto the original feasible set. The {\L}ojasiewicz inequality then converts the error in the defining constraints into an error in the moments. With the {\L}ojasiewicz exponent $0<\Lo\leq1$, we obtain the convergence rates $\mathrm{O}((\log_2 r)^{3\Lo/2}/r^{\Lo})$ for the Putinar-type hierarchy and $\mathrm{O}(1/r^{\Lo/2})$ for the normalized Krivine--Stengle-type and Extended-Handelman-type hierarchies. In the matrix setting, we utilize Chebyshev-type kernel on $[-1,1]^n$ to derive the rate $\mathrm{O}((\log_2 r)^{3\Lo/2}/r^{\Lo})$ for the Hol-Scherer-type hierarchy. Together with our preceding work, the results provide a universal method for studying convergence rates of different types of moment-SOS hierarchies.
\end{abstract}
\section{Introduction}
The paper considers the problem of minimizing a polynomial $f \in \RR[\bx]$ over a nonempty compact basic semialgebraic set $\CX \subset \RR^n$:
\begin{equation}\tag{POP}\label{POP}
\fmin = \underset{\bx \in \CX}{\min}\ f(\bx).
\end{equation}
The semialgebraic set $\CX$ is defined either by polynomial inequalities and equalities as follows:
\begin{equation}\label{def: poly_set}
\CX := \big\{ \bx \in \RR^n:\
g_j(\bx) \geq 0 \ \forall j \in [m],\;
h_i(\bx) = 0 \ \forall i \in [p] \big\},
\end{equation}
where $g_j,\,h_i \in \RR[\bx]$ for all $j \in [m]$ and $i \in [p]$, or by a polynomial matrix inequality of the form
\begin{equation}\label{def: matrix_set}
\CX := \big\{ \bx \in \RR^n:\ G(\bx)\succeq 0 \big\},
\end{equation}
where $G(\bx)=(g_{ij}(\bx))_{m \times m}$ is an $m \times m$ symmetric polynomial matrix with $g_{ij} \in \RR[\bx]$. Polynomial optimization problems (POPs) of these two forms are very general and have wide applications in various fields. POPs in the scalar setting, where $\cX$ is defined as in~\eqref{def: poly_set}, have applications in operations research, signal processing, and computational geometry. We refer to~\cite{lasserre2009moments,henrion2020moment,nie2023moment,magron2023sparse,tran2025momentsumofsquareshierarchygromov} for overviews of existing techniques and applications. POPs in the matrix setting have applications in various fields, including control theory~\cite{chesi2010lmi,henrion2006convergent,ichihara2009optimal,scherer2006lmi}, quantum information theory~\cite{fang2021sum}, and statistics~\cite{henrion2020moment}.

The problem~\eqref{POP} is NP-hard, as it includes hard combinatorial optimization problems such as MAX-CUT (see, e.g.,~\cite{michael1979johnson}) as special cases. To address this computational difficulty, the moment-SOS hierarchy, introduced in the works of Lasserre~\cite{lasserre2001global} and Parrilo~\cite{parrilo2000structured}, has become a powerful method for solving POPs. The basic idea underlying this method is that~\eqref{POP} admits the reformulation
\begin{displaymath}
    \fmin =\; \underset{\tau \in \RR}{\sup} 
    \big\{\tau :\; f-\tau \in \CP_+(\CX)\big\},
\end{displaymath}
where $\CP_+(\CX)$ denotes the set of nonnegative polynomials over $\CX$. This establishes a connection between POPs and the problem of verifying the non-negativity (or positivity) of a polynomial over a semialgebraic set $\CX$. The connection shows that we can solve POPs by providing inner approximations of $\CP_+(\CX)$, whose membership can be verified efficiently. Therefore, different {\em certificates of positivity} induce different hierarchies of lower bounds for $\fmin$, which can be computed using semidefinite programming (SDP) or linear programming (LP). We summarize the types of hierarchies whose convergence rates we study in the following section.

\subsection{Positivity certificates and hierarchy of lower bounds}
Throughout the paper, the relaxation level is denoted by $r$. The SOS hierarchies use certificates of degree at most $2r$, whereas the Krivine--Stengle-type and Extended-Handelman-type hierarchies use certificates of degree at most $r$. The subscripts on truncated cones always indicate the total certificate degree.
For $\bx \in \RR^n$, we define the vector of monomials in $\bx$ by 
\begin{displaymath}
    \bv(\bx) = (1\, \, x_1 \ldots x_n \ldots \bx^\alpha \ldots)^\top,
\end{displaymath}
where $\alpha \in \NN^n$ is a multi-index used to define monomial $\bx^\alpha = \prod_{i=1}^n x_i^{\alpha_i}$. For any fixed positive integer $r$, we denote the vector of monomials of degree at most $r$ by $\bv_r(\bx)$. One of the most natural positivity certificates of a polynomial in $\bx$ is checking {\em sum-of-squares (sos)}. We say a polynomial $p(\bx)$ is sum-of-squares if there exist polynomials $p_1,\ldots,p_k$ such that $\sum_{i=1}^kp_i^2 = p$. We denote the set of such polynomials by $\Sigma[\bx]$ and let $\Sigma_r[\bx]$ be the subset consisting of all sos--polynomials of degree at most $r$. Obviously, a sos--polynomial is non-negative. Moreover, for any sos--polynomial $p(\bx)$ of degree at most $2r$, there exists a positive semi-definite matrix $M$ such that 
\begin{displaymath}
    p(\bx) \, = \, \langle M, \bv_r(\bx)\bv_r(\bx)^\top \rangle,
\end{displaymath}
so membership of $\Sigma_{2r}[\bx]$ can be checked via SDP. 
However, $\Sigma[\bx]$ is not sufficient to cover all non-negative polynomials.
A classical example of a nonnegative polynomial that is not a sum of squares is the Motzkin polynomial~\cite{motzkin1967arithmetic}, defined by
\begin{displaymath}
    M(x,y):=x^4y^2+x^2y^4+1-3x^2y^2.
\end{displaymath}
Indeed, by the arithmetic--geometric mean inequality,
\begin{displaymath}
    x^4y^2+x^2y^4+1 \geq 3\sqrt[3]{x^4y^2\cdot x^2y^4} =3x^2y^2.
\end{displaymath}
Therefore, $M(x,y)\geq 0$ for all $(x,y)\in\RR^2$. Nevertheless, $M$ cannot be represented as a sum of squares of polynomials.
Fortunately, for checking the non-negativity of a polynomial over a compact semialgebraic set $\cX$ of either form~\eqref{def: poly_set} or~\eqref{def: matrix_set}, the SOS method can be strengthened by using a {\em preordering} or a {\em quadratic module} generated by the constraints defining $\cX$.

In the scalar setting, where $\cX$ is defined by polynomial inequalities and equalities as in~\eqref{def: poly_set}, the truncated preordering and quadratic module of $\CX$ are defined, respectively, by 
\begin{eqnarray} 
    \CT(\CX)_{2r} &=& \Biggl\{q(\bx) = \sum_{j=1}^N\sigma_{J_j}(\bx)g_{J_j}(\bx) + \sum_{i=1}^p\tau_i(\bx)h_i(\bx) :\ 
     \label{eq-TX} \\ 
    &&\quad \exists N \in \NN,\; J_j\subset [m], \;\sigma_{J_j} \in \Sigma[\bx]_{2(r-\lceil g_{J_j} \rceil) } \ \forall j \in [N],\ \tau_i \in \RR[\bx]_{2r-\deg h_i}\ \forall i \in [p] \ \Biggr\},
\nonumber\\
        \CQ(\CX)_{2r}&=&\Biggl\{q(\bx) = \sigma_0(\bx) +\sum_{j=1}^m\sigma_j(\bx)g_j(\bx) 
        +\sum_{i=1}^p\tau_i(\bx)h_i(\bx):
        \label{eq-QX} \\ 
    &&\quad \sigma_0 \in \Sigma[\bx]_{2r},\;
     \; \sigma_j \in \Sigma[\bx]_{2(r-\lceil g_j \rceil)}\ \forall\; j\in[m],\;
     \tau_i \in \RR[\bx]_{2r-\deg h_i}\ \forall i \in [p]\Biggr\}.
    \nonumber
\end{eqnarray}
Here, for an index set $ J \subset [m]$, we define $g_J = \prod_{j \in J}g_j$,
 $g_{\emptyset}=1$, and $\degg{p} := \lceil \deg(p)/2 \rceil$. We also assume that $\degg{g_{J_j}} \leq r$ for all $j\in[N]$, $\degg{g_j}\leq r$ for all $j\in [m]$, 
and $\degg{h_i}\leq r$ for all $i\in[p]$ in~\eqref{eq-QX} and~\eqref{eq-TX}. The preordering $\CT(\CX)$ and quadratic module $\CQ(\CX)$ are defined by removing the degree constraints on component polynomials of~\eqref{eq-TX} and~\eqref{eq-QX}, respectively. One can observe that we slightly abuse the notations, where $\CT(\CX)$ and $\CQ(\CX)$ are defined by the polynomials that define $\cX$ rather than the set $\cX$ itself. In this paper, we will explicitly state the polynomials used in definition of preordering and quadratic module.

It is clear that $\CT(\CX)_{2r}$ and $\CQ(\CX)_{2r}$ provide inner approximations of $\CP_+(\cX)$, and membership of these truncated preordering and truncated quadratic module is verifiable via SDPs. This results in the following hierarchies of lower bounds on the global optimal value $\fmin$ of~\eqref{POP}:
\begin{eqnarray}
    \lb(f,\CT(\CX))_r = \sup\{\tau \in \RR:\ f(\bx) - \tau \in \CT(\CX)_{2r}\} \label{hierarchy: poly-preordering},\\
     \lb(f,\CQ(\CX))_r = \sup\{\tau \in \RR:\ f(\bx) - \tau \in \CQ(\CX)_{2r}\} \label{hierarchy: poly-quadratic}.\,
\end{eqnarray}
The convergence of these hierarchies is the consequence of the Schmüdgen Positivstellensatz for a compact semi-algebraic set~\cite[pp. 283--313]{schmudgen2017moment}
and the Putinar Positivstellensatz for an Archimedean semi-algebraic set~\cite{putinar1993positive}, respectively. 

\begin{theorem}[Schmüdgen Positivstellensatz]\label{thm: Schmüdgen}
    Let $\CX$ be the semi-algebraic set in~\eqref{def: poly_set}.
    We assume that $\CX$ is compact. If $f$ is a positive polynomial on $\CX$, then $f \in \CT(\CX)$.
\end{theorem}
\begin{theorem}[Putinar Positivstellensatz]\label{thm: Putinar}
    Let $\CX$ be the semi-algebraic set in~\eqref{def: poly_set}. We assume that the Archimedean condition holds, i.e., there exists a positive number $R$ such that $R -\|\bx\|^2 \in \CQ(\CX)$. If $f$ is a positive polynomial on $\CX$, then $f \in \CQ(\CX).$
\end{theorem}

Whereas the Positivstellensatze of Schm\"udgen and Putinar induce hierarchies of semidefinite programs, the Krivine--Stengle-type and Extended-Handelman-type hierarchies are formulated as linear programs. Two prominent examples are the Krivine--Stengle and extended Handelman Positivstellensatze. These results provide inner approximations of $\CP_+(\cX)$ through sequences of truncated preprimes generated by the polynomials defining $\cX$ in~\eqref{def: poly_set}.
\begin{multline}\label{def: preprime}
 \CR(\cX)_r:=\Big\{\sum_{\alpha,\beta\in\NN^m}
 \sigma_{\alpha,\beta}\bg^\alpha(1-\bg)^\beta+\sum_{i=1}^p\tau_i h_i:
 \ \sigma_{\alpha,\beta}\geq0,\ \tau_i\in\RR[\bx],\\
 \deg(\bg^\alpha)+\deg((1-\bg)^\beta)\leq r,\quad
 \deg\tau_i+\deg h_i\leq r\Big\},
\end{multline}
where $\bg = (g_1,\dots,g_m)$, $\bg^\alpha = \prod_{j=1}^mg_j^{\alpha_j}$, and $(1-\bg)^\beta = \prod_{j=1}^m(1-g_j)^{\beta_j}$. The preprime $\CR(\CX)$ associated to $\cX$ is defined as in~\eqref{def: preprime} after lifting all constraints on degree of component polynomials. 
The truncated preprime yields the following convex relaxation of~\eqref{POP}:
\begin{equation}\label{hierarchy: LP}
     \lb(f,\CR(\CX))_r = \sup\{\tau \in \RR:\ f(\bx) - \tau \in \CR(\CX)_{r}\}.
\end{equation}
Since the coefficients $\sigma_{\alpha,\beta}$ are nonnegative scalars and all remaining constraints are linear coefficient-matching conditions, the relaxation~\eqref{hierarchy: LP} can be formulated as an LP. Its asymptotic convergence follows from the Krivine--Stengle Positivstellensatz.
\begin{theorem}
    Assume $1- g_j(\bx) \geq 0$ on $\cX$, which can be achieved by a suitable normalization on $\bg$, and $\{1,g_1,\ldots,g_m,h_1,\ldots,h_p\}$ generates $\RR[\bx]$. Then $f \in \CR(\CX)$ for any $f(\bx) > 0$ on $\cX$.
\end{theorem}
This theorem was first introduced by Krivine in~\cite{krivine1964anneaux,krivine1964quelques}, and later revisited in the works~\cite{becker1983darstellungssatz,marshall2002general,vasilescu2003spectral}. If the set $\{1,g_1,\ldots,g_m\}$ fails to generate $\RR[\bx]$, one can assume that $\cX$ is contained in the hypercube $[0,1]^n$ by suitable normalization and replace the set of polynomials defining $\cX$ by 
\begin{displaymath}
    \cX := \{\bx \in \RR^n:\; x_i \geq 0,\; 1-x_i \geq 0 \, \forall\, i \in [n],\, g_j(\bx) \geq 0 \, \forall\, j \in [m],\, h_i(\bx) = 0 ,\; \forall\, i \in [p]\}.
\end{displaymath}
Then the $r$--truncated preprime is defined by 
\begin{multline}\label{def: extended preprime}
 \CR(\cX)_r:=\Big\{\sum_{\alpha\in\NN^m,\,\beta,\gamma\in\NN^n}
 \sigma_{\alpha,\beta,\gamma}\bg^\alpha\bx^\beta(1-\bx)^\gamma+\sum_{i=1}^p\tau_i h_i:
 \ \sigma_{\alpha,\beta,\gamma}\geq0,\ \tau_i\in\RR[\bx],\\
 \deg(\bg^\alpha)+|\beta|+|\gamma|\leq r,\quad
 \deg\tau_i+\deg h_i\leq r\Big\},
\end{multline}
for which the sequence $\{\lb(f,\CR(\CX))_r\}$ still converges to the global optimal value $\fmin$ according to Extended-Handelman’s Positivstellensatz that is first introduced in~\cite{handelman1988representing} and restated in the compact case in~\cite{kuryatnikova2024reducing}. The theorem is stated as follows: 
\begin{theorem}~\cite[Proposition~3.5]{kuryatnikova2024reducing}
    Assume $\cX$ is contained in $[0,1]^n$ and the truncated preprime is defined as in~\eqref{def: extended preprime}. Then $f \in \CR(\CX)$ for any $f(\bx) > 0$ on $\cX$.
\end{theorem}

In the matrix setting, where $\cX$ is defined by a matrix inequality as in~\eqref{def: matrix_set}, an analogue of the previous framework in the scalar setting was established by Hol and Scherer~\cite{hol2004sum}. In particular, they generalized Putinar's Positivstellensatz to semialgebraic sets of the form~\eqref{def: matrix_set}.

Extending from sos--polynomials, we say that a symmetric $m \times m$ polynomial matrix $P(\bx)$ is a {\em sos-polynomial matrix} if there exists a (not necessarily square) polynomial matrix $T(\bx)$ such that $P(\bx) = T(\bx)^{\top}T(\bx)$. We denote the set of all $m \times m$ sos--polynomial matrix by $\bS\Sigma[\bx]^m$. We define the {\em matrix quadratic module} $\cQ(\cX)$ associated with $\cX$ in \eqref{def: matrix_set} as: 
\begin{equation*} 
    \cQ(\mathcal{X}):= \left\{ \sigma(\bx) + \langle R(\bx),G(\bx)\rangle \;:\; \sigma \in \Sigma[\bx],\ R(\bx) \in \bS\Sigma[\bx]^m \right\},
\end{equation*} 
and its truncation at degree $2r$, defined as follows: 
\begin{multline*}
    \cQ(\cX)_{2r}:= \Bigl\{ \sigma(\bx) + \langle R(\bx),G(\bx)\rangle\;: \sigma \in \Sigma[\bx],\ R(\bx) \in \bS\Sigma[\bx]^m,\\ \deg \sigma \leq 2r,\; \deg R \leq 2(r-\lceil G\rceil) \Bigr\},
\end{multline*}
where $\lceil G\rceil:=\lceil\deg G/2\rceil$. Since $R$ is an SOS matrix, its degree is even, so the multiplier condition is equivalent to $\deg R+\deg G\leq2r$. Membership can be checked by an SDP (see e.g., \cite{hol2004sum,hol2005sum,henrion2006convergent}). As a result, Hol and Scherer proposed the following hierarchy of lower bounds to approximate the optimal value $f_{\min}$ of $f$ over $\cX$:
\begin{equation}\label{hierarchy: Hol and Scherer}
    \lb(f,\cQ(\cX))_r :=\; \sup\{t \in \bR\;:\; f- t \in \cQ(\cX)_{2r}\}.
\end{equation}
Under the Archimedean condition on $\cX$, the lower bound $\lb(f,\cQ(\cX))_r$ converges to $f_{\min}$ as $r \to \infty$, which is secured by a natural extension of Putinar's Positivstellensatz \cite{putinar1993positive}, as stated in the following theorem.

\begin{theorem}\cite[Theorem 1]{hol2004sum}\label{GPutinar} 
Suppose $\mathcal{X}$ satisfies the Archimedean condition, that is, there exist an SOS polynomial $\sigma$, an SOS polynomial matrix $R(\bx)$, and a scalar $R$ such that 
\begin{displaymath}
    R - \bx^\top\bx  = \sigma(\bx) +  \langle R(\bx),G(\bx) \rangle.
\end{displaymath}
Then every positive polynomial $f$ on $\mathcal{X}$ belongs to the quadratic module $\CQ(\mathcal{X})$.
\end{theorem}
For better representation, in this paper, we name the hierarchies~\eqref{hierarchy: poly-preordering},~\eqref{hierarchy: poly-quadratic},~\eqref{hierarchy: LP}, and~\eqref{hierarchy: Hol and Scherer} the Schm\"udgen-type hierarchy, the Putinar-type hierarchy, the Krivine--Stengle or extended-Handelman-type hierarchy, and the Hol-Scherer-type hierarchy.

\subsection{Related work} A natural question is how efficient the aforementioned hierarchies are, which is interpreted as how fast their lower bounds converge to the global optimal value $\fmin$ of~\eqref{POP}. We outline here the related works for all aforementioned certificates of positivity, starting with convergence rate related results in the scalar case as $\CX$ is defined as in~\eqref{def: poly_set}.

The Schm\"udgen-type, or preordering-type, moment--SOS hierarchy admits some of the sharpest known convergence rates. When the feasible set $\cX$ is a simple set, namely, the unit ball, the hypercube, or the standard simplex, the works~\cite{slot2111sum,laurent2023effective,S.o.S-on-simplex} develop techniques based on the Christoffel--Darboux (CD) kernel to approximate positive polynomials by SOS polynomials. These techniques prove an explicit convergence rate of $\mathrm{O}(1/r^2)$. For the unit sphere $\mathbb{S}^{n-1}$, the same convergence rate was established in~\cite{fang2021sum} for homogeneous polynomial objectives and subsequently extended to general polynomial objectives in~\cite{blomenhofer2024moment}. For the binary hypercube $\cX=\{0,1\}^n$, it follows from~\cite{fawzi2016,sakaue2017} that the corresponding Schm\"udgen-type moment--SOS hierarchy achieves finite convergence. More precisely, for a polynomial objective of degree $d$, the hierarchy is exact whenever $r\geq\lceil(n+d-1)/2\rceil$. More recently, the rate $\mathrm{O}(1/r^2)$ for products of simple sets was established in~\cite{magron2025convergence}.

The Putinar-type (quadratic-module-type) moment-SOS hierarchy is weaker than the Schm\"udgen-type hierarchy, but is widely used in practice because it requires significantly fewer localizing matrices. When $\cX$ is the unit ball, the method based on the CD kernel gives the rate $\mathrm{O}(1/r^2)$~\cite{slot2111sum}. For the hypercube, Gribling et al.~\cite{gribling2025revisitingconvergenceratelasserre} studied the quadratic-module hierarchy; their more recent squared-kernel construction gives the rate $\mathrm{O}(\log_2^3(r)/r^2)$~\cite{gribling2026squared}.

Under the Archimedean condition, and for the Putinar-type hierarchy, Nie and Schweighofer~\cite{Putinar-complexity} provided a logarithmic convergence rate. Polynomial degree estimates were subsequently established by Baldi and Mourrain~\cite{moment-approximation}. Constructing penalty functions has become one of the main methods for studying these rates; see~\cite{moment-approximation,Baldi_2025,heijmans2026degree}. To compare their exponents, we use the convention
\begin{displaymath}
    \bd(\bx,\cX)\leq c_{\Lo}\bp(\bx)^{\Lo},\qquad0<\Lo\leq1.
\end{displaymath}
 
When a reference writes its error bound as $\bd(\bx,\cX)^L\leq c\bp(\bx)$, its exponent satisfies $L=1/\Lo$.

Baldi et al.~\cite{Baldi_2025} obtained a degree exponent $7L+3$. In the above convention, this gives an objective-value rate $\mathrm{O}(1/r^{\Lo/(7+3\Lo)})$, and $\mathrm{O}(1/r^{1/10})$ under their constraint qualification condition (CQC). Heijmans-Kuryatnikova et al.~\cite{heijmans2026degree} give a degree exponent $2L$ for Putinar-type certificates, leading to $\mathrm{O}(1/r^{\Lo/2})$. Their construction uses lifting, with the analysis carried out at the level of positivity certificates. Combining this construction with the improved hypercube estimate gives exponent $\Lo$; Gribling et al.~\cite{gribling2026squared} state the resulting general-set rate as $\mathrm{O}(\log_2^3(r)/r^{\Lo})$. Other works studying lower-bound convergence include~\cite{deKlerkLaurent2010,S.o.S-on-simplex,Putinar-complexity}. In the univariate case, Henrion and Safey El Din~\cite{henrion2026univariate} obtain $\mathrm{O}(1/r^2)$ under boundedness, using the special structure of one-dimensional feasible sets.

For the Krivine--Stengle-type and Extended-Handelman-type hierarchies, convergence rates have also been studied. On simple sets such as the hypercube and simplex, the classical Bernstein-polynomial estimates give $\mathrm{O}(1/r)$; see, e.g.,~\cite{deKlerkLaurent2010,S.o.S-on-simplex}. Effective degree bounds for these Positivstellens\"atze on general compact semi-algebraic sets were established in~\cite{heijmans2026degree}, leading to $\mathrm{O}(1/r^{\Lo/2})$ in the above convention.

In the matrix setting, where $\cX$ is defined as in~\eqref{def: matrix_set}, the Hol-Scherer-type hierarchy was introduced in~\cite{hol2004sum,hol2005sum,Kojima}. Tran et al proved a polynomial rate in the work\cite{tran2024convergence}, and Huang~\cite{huang2025complexity} obtained the convergence rate $\mathrm{O}(1/r^{1/(7\eta+3)})$, where $\eta$ is the {\L}ojasiewicz exponent of an auxiliary scalar description. This exponent cannot in general be identified with $1/\Lo$ for the direct spectral residual.

\subsection{Contribution}
Our main contribution is to extend the method introduced in~\cite{tran2025convergence}, which proves $\mathrm{O}(1/r^{\Lo})$ for its Schm\"udgen-type construction on compact sets, to the other hierarchies described above. We first compare the resulting rates with the best rates recalled in the preceding subsection. Our estimates concern the Hausdorff distance between truncated pseudo-moment sets and moment sets; by Lemma~\ref{lemma: distance to convergence rate}. Therefore we provide uniform convergence rate over polynomials of a fixed degree.

For the Putinar-type hierarchy on $\mathrm H^n$, the best available bound is $\mathrm{O}(\log_2^3(r)/r^2)$~\cite{gribling2026squared}. Theorem~\ref{thm: Hausdorff over H^n} recovers this rate for the Hausdorff distance. On general sets, the strongest exponent supplied by the preceding results is $\Lo$, obtained by combining~\cite{heijmans2026degree} with the improved hypercube estimate. Theorem~\ref{thm: convergence rate Putinar-type} recovers this exponent with the bound $\mathrm{O}((\log_2 r)^{3\Lo/2}/r^{\Lo})$. Thus it improves the power of $r$ compared with the earlier $\mathrm{O}(1/r^{\Lo/2})$ bound, while giving a moment-approximation proof of the power already available from the combined results.

For the normalized Krivine--Stengle-type and Extended-Handelman-type hierarchies, the best rates recalled above are $\mathrm{O}(1/r)$ on $\mathrm H^n$ and $\mathrm{O}(1/r^{\Lo/2})$ on general compact sets~\cite{deKlerkLaurent2010,heijmans2026degree}. Theorems~\ref{thm: bound on Hausdorff Handelman over hypercube} and~\ref{thm: convergence rate Handelman} recover these rates for the Hausdorff distance. The contribution in this case is the common moment-approximation argument and its uniform estimate, with the same exponent of $r$.

For the Hol-Scherer-type hierarchy, the existing bound $\mathrm{O}(1/r^{1/(7\eta+3)})$~\cite{huang2025complexity} uses an auxiliary scalar error bound. Our preceding work~\cite[Theorem 5.1(i)]{tran2024convergence} gives $\mathrm{O}(1/r^{2\Lo/(n+2\Lo+5)})$ for a reduced matrix hierarchy. Theorem~\ref{thm:matrix} gives $\mathrm{O}((\log_2 r)^{3\Lo/2}/r^{\Lo})$ for the Hol-Scherer-type hierarchy on $[-1,1]^n$ with explicit generators $1-x_i^2$. Since $\Lo>2\Lo/(n+2\Lo+5)$, this improves the power in the preceding reduced-hierarchy bound. Relative to Huang's bound, the power improves when $\Lo>1/(7\eta+3)$, agrees when equality holds, and is smaller otherwise; equality of powers alone does not remove the logarithmic factor. These comparisons retain the respective generator assumptions and error-bound conventions. Here $\Lo$ is associated with $\max\{0,-\lambda_{\min}(G(\bx))\}$, so no ordering of $\Lo$ and $1/(7\eta+3)$ is assumed.

Table~\ref{tab: convergence exponents} collects the general-set rates for fixed problem data and moment degree, under the assumptions of the respective results. Each rate is written using $\mathrm{O}(1/r^\gamma)$, with its logarithmic factor displayed separately. The existing column concerns objective-value bounds, while the bounds in this paper hold for the uniform moment distance and hence also for objective values.
\begin{table}[!htb]
\centering
\small
\setlength{\tabcolsep}{5pt}
\renewcommand{\arraystretch}{1.18}
\begin{tabular}{@{}>{\raggedright\arraybackslash}p{4.0cm}>{\raggedright\arraybackslash}p{5.45cm}>{\raggedright\arraybackslash}p{4.75cm}@{}}
\hline
Hierarchy & Existing bound & This paper \\
\hline
Putinar-type
& $\mathrm{O}(1/r^{\Lo/2})$\newline
~\cite{heijmans2026degree}\newline
& $(\log_2 r)^{3\Lo/2}\,\mathrm{O}(1/r^{\Lo})$\newline
Theorem~\ref{thm: convergence rate Putinar-type} \\[4pt]
\hline
Krivine--Stengle-type /\newline Extended-Handelman-type
& $\mathrm{O}(1/r^{\Lo/2})$\newline
\cite{heijmans2026degree}
& $\mathrm{O}(1/r^{\Lo/2})$\newline
Theorem~\ref{thm: convergence rate Handelman} \\[4pt]
\hline
Hol-Scherer-type
& $\mathrm{O}(1/r^{1/(7\eta+3)})$\newline
\cite{huang2025complexity}
& $(\log_2 r)^{3\Lo/2}\,\mathrm{O}(1/r^{\Lo})$\newline
Theorem~\ref{thm:matrix} \\
\hline
\end{tabular}
\caption{Convergence rates on compact semialgebraic sets.}
\label{tab: convergence exponents}
\end{table}

\subsection{Methodology: outline of the proof technique}\label{sec: method}
We now explain the method presented via 2 main stages. The first stage establishes a base error $\varepsilon_r$ from an estimate on a simple set. The second stage uses this estimate after a series of lifting and sequences to control the Hausdorff distance of truncated pseudo-moment sequence, which measures how far a truncated pseudo-moment sequence to the set of actual truncated moment sequences.


\paragraph{Stage 1: Base estimation.}
For a fixed relaxation order, let $\cS(\cX)$ denote one of the truncated certificates $\CT(\cX)_{2r}$, $\CQ(\cX)_{2r}$, or $\CR(\cX)_r$. We establish compactness of the corresponding normalized pseudo-moment set and equality of its primal and strong duality between lower bounds given via SOS presentation and moment sequence presentation. We use the existing convergence rate of $[0,1]^n$ in the scalar cases, $[-1,1]^n$ in the matrix case to provide these properties, and set up the base error.

In the scalar cases, we use the base convergence rates $\mathrm{O}(\log_2^3(r)/r^2)$ for the quadratic module over $[0,1]^n$ and $\mathrm{O}(1/r)$ for the Extended-Handelman-type preprime over $[0,1]^n$ to delivery on the Hausdorff distance.

In the matrix case, we use the base convergence rates $\mathrm{O}(\log_2^3(r)/r^2)$ for the quadratic module over $[0,1]^n$ the squared-kernel in Section~\ref{sec:matrix} to yield a bound on the Hausdorff distance and and a lower bound for spectrum of $G$.

\paragraph{Stage 2: Lift and projection.}
For scalar inequalities, introduce auxiliary variables $u_j=g_j(\bx)$ after normalizing $\mon{g_j}\leq1$ to eliminate polynomial inequalities of the form $g_j(\bx) \geq 0$. The set
\begin{displaymath}
   \cK=\{(\bx,\bu)\in\mathrm{H}^{n+m}:h_i(\bx)=0,\; g_j(\bx)-u_j=0\} 
\end{displaymath}
 
projects onto $\cX$. 

The following map on sequences defined by substitution can be used to lift pseudo-moment sequence on $\cX$ to pseudo-moment sequence on $\cK$.
\begin{displaymath}
   \ell_{\biy^{2r,t}}(\bx^\alpha\bu^\beta) =\ell_{\biy}(\bx^\alpha\bg(\bx)^\beta).
\end{displaymath}

We then project the lifted sequence onto $\CM_t(\mathrm{H}^{n+m})$. The base estimate controls its distance, and the squared equations $h_i^2$ and $(g_j-u_j)^2$ control the average squared violation of its atoms. After retaining the $\bx$-moments,~\eqref{eq: universal moment transfer} gives the desired bound. Figure~\ref{fig: method lifting} illustrates this construction. The upper arrow lifts the pseudo-moment sequence to the auxiliary space; after moment approximation there, the lower arrow projects the atoms onto $\cX$ and retains their degree-$k$ moments. Here $\CY=\mathrm H^n$ and $\CZ=\mathrm H^{n+m}$ denote the original and lifted ambient sets, and $\biy^\varphi$ denotes the lifted sequence.
\begin{figure}[!ht]
    \centering
\begin{subfigure}{0.41\textwidth}
    \begin{adjustbox}{max width=\linewidth}
\begin{tikzpicture}

    \begin{scope}[blend mode=multiply]
    \filldraw[fill=gray!10, draw=black] (1,0) ellipse (3.3cm and 1.5cm);
    \filldraw[fill=white!10, draw=black] (0,0) circle (1.2cm);
    \filldraw[fill=gray!10, draw=black] (0,0) ellipse (1.7cm and 2.7cm);
    \filldraw[fill=white!10, draw=black] (1,0) circle (3.5cm);
    \end{scope}
    
    \node at (0,0) {$\CM_{k}(\CX)$};
    \node at (3.0,0) {$\CM_k(\CY)$};
    \node at (0,1.8) {$\mathcal{P\!M}_k(\cS(\cX))$};
    \node at (1.2,-3.0) {$\mathcal{P\!M}_k(\cS(\CY))$};
    \node at (0,-2.3) (S1) {$\biy$};
    \node at (2.7,-0.8) (S2) {$\overline{\biy}$};
    \node at (0,-0.7) (S3) {$\widetilde{\biy}$};

    \draw[arrow] (S1.east) -- (S2.west);
    \draw[arrow] (S2.west) -- (S3.south);
    \draw[arrow] (S1.east) -- (S3.south);
\end{tikzpicture}
\end{adjustbox}
\end{subfigure}
\hfill
\quad\tikz[overlay,remember picture] 
{\draw[thick, ->] (-1.7,3.5) --node[above]{$\operatorname{\mathbf{Lifting}}$}node[below]{\textbf{via }$\varphi$} (0.5,3.5);
\draw[thick, <-] (-1.7,1.5) --node[above]{$\operatorname{\mathbf{Projection}}$}node[below]{\textbf{via atoms }} (0.5,1.5);}
\quad
  \begin{subfigure}{0.41\textwidth}
    \begin{adjustbox}{max width=\linewidth}
\begin{tikzpicture}

    \begin{scope}[blend mode=multiply]
    \filldraw[fill=gray!10, draw=black] (1,0) ellipse (3.3cm and 1.3cm);
    \filldraw[fill=white!10, draw=black] (0,0) circle (1.2cm);
    \filldraw[fill=gray!10, draw=black] (0,0) ellipse (1.7cm and 2.6cm);
    \filldraw[fill=white!10, draw=black] (1,0) circle (3.5cm);
    \end{scope}
    
    \node at (0,0) {$\CM_{t}(\cK)$};
    \node at (3.0,0) {$\CM_t(\CZ)$};
\node at (0,1.6) {$\mathcal{P\!M}_t(\cS(\cK))$};
    \node at (1.0,-2.9) {$\mathcal{P\!M}_t(\cS(\CZ))$};
    \node at (0,-2.3) (S1) {$\biy^{\varphi}$};
    \node at (2.7,-0.8) (S2) {$\overline{\biy}^\prime$};
    \node at (0,-0.5) (S3) {$\widetilde{\biy}^\prime$};

    \draw[arrow] (S1.east) -- (S2.west);
    \draw[arrow] (S2.west) -- (S3.south);
    \draw[arrow] (S1.east) -- (S3.south);
\end{tikzpicture}
\end{adjustbox}
\end{subfigure}

\caption{Lifting to the auxiliary moment space and projection via atoms.}
\label{fig: method lifting}
\end{figure}
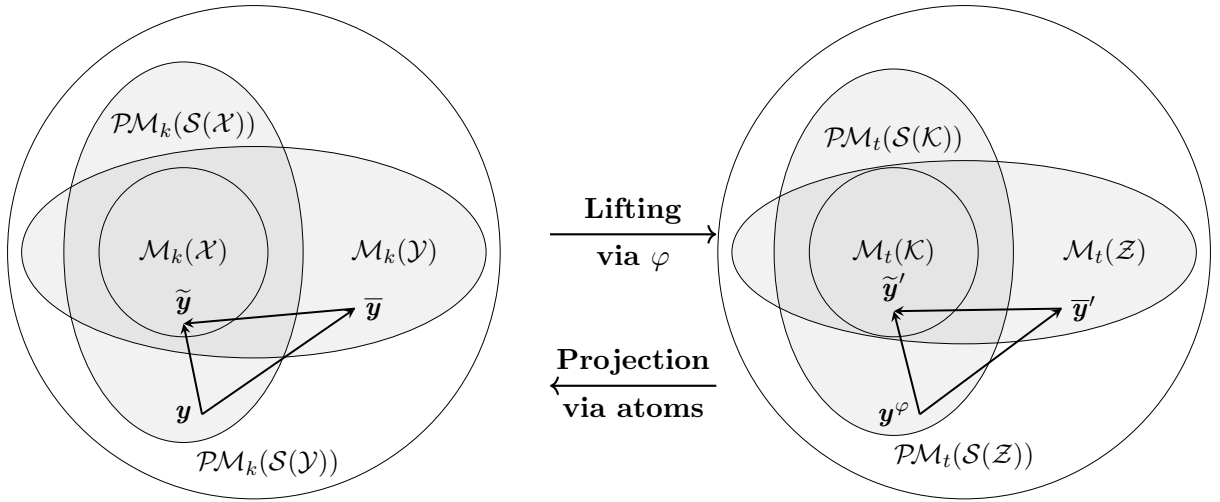

For matrix inequalities, we use the same idea with some modification. The distance to the set of truncated moment sequences now depends on the smallest eigenvalue function 
\begin{displaymath}
    \max\{0,-\lambda_{\min}(G(\bx))\},
\end{displaymath}
which in general is not a polynomial. Handling this function needs to be subtle, and we utilize the squared-kernel defined over $[-1,1]^n$ to resolve this concern. 

To complete Stage 2, we repeatedly use the following common estimate on atoms. Each technical section then only needs to bound the moment displacement and the violation. Let $\CY=[0,1]^n$ in the scalar cases and $\CY=[-1,1]^n$ in the matrix case. Suppose $\cX$ is a nonempty compact subset of $\CY$, and let $\bp:\CY\to\RR_{\geq0}$ satisfy
\[
 \bd(\bx,\cX)\leq c_{\Lo}\bp(\bx)^{\Lo},\qquad 0<\Lo\leq1.
\]
Write $\bv_k$ for the vector of monomials of degree at most $k$, and let $\Lip_k(\CY)$ be a Lipschitz constant of this vector on $\CY$. Suppose a given sequence $\biy$ is approximated by
\[
 \overline{\biy}=\sum_{s=1}^Nw_s\bv_k(\bx_s),\qquad
 \bx_s\in\CY,\quad w_s\geq0,\quad\sum_{s=1}^Nw_s=1.
\]
Choose closest points $\widetilde{\bx}_s\in\cX$ and put $\widetilde{\biy}=\sum_sw_s\bv_k(\widetilde{\bx}_s)\in\CM_k(\cX)$. The triangle inequality, the error bound, and Jensen's inequality give
\begin{equation}\label{eq: atom transfer}
\begin{split}
 \|\overline{\biy}-\widetilde{\biy}\|
 &\leq\Lip_k(\CY)\sum_sw_s\bd(\bx_s,\cX)\\
 &\leq c_{\Lo}\Lip_k(\CY)\sum_sw_s\bp(\bx_s)^{\Lo}\\
 &\leq c_{\Lo}\Lip_k(\CY)
       \left(\sum_sw_s\bp(\bx_s)^2\right)^{\Lo/2}.
\end{split}
\end{equation}
The last inequality uses the concavity of $a\mapsto a^{\Lo/2}$ and the identity $\sum_sw_s=1$. Consequently, whenever
\[
 \|\biy-\overline{\biy}\|\leq e_r,\qquad
 \sum_sw_s\bp(\bx_s)^2\leq\eta_r,
\]
we obtain the common bound
\begin{equation}\label{eq: universal moment transfer}
 \bd(\biy,\CM_k(\cX))
 \leq e_r+c_{\Lo}\Lip_k(\CY)\eta_r^{\Lo/2}.
\end{equation}
No factor depending on the number of atoms is introduced. A representing measure in a lifted space supplies the same estimate by retaining its weights and its $\bx$-coordinates.

In each section, the two quantities to estimate are therefore $e_r$ and $\eta_r$. If the base error is $\varepsilon_r$, the scalar lifts give $e_r=\mathrm{O}(\varepsilon_r)$ and $\eta_r=\mathrm{O}(\varepsilon_r)$, while the matrix projection with unit penalty gives $e_r=\mathrm{O}(\sqrt{\varepsilon_r})$ and $\eta_r=\mathrm{O}(\varepsilon_r)$. Since $0<\Lo\leq1$, both lead to $\mathrm{O}(\varepsilon_r^{\Lo/2})$. Finally, Lemma~\ref{lemma: distance to convergence rate} converts the Hausdorff bound into
\[
 f_{\min}-\mlb(f,\cS(\cX))
 \leq\monc{f}\,\bd_k(\cS(\cX)).
\]
The certificates and degree estimates depend on the hierarchy; the final atom estimate is common to all three cases.

\section{Preliminaries}
\subsection{Notation}

We denote by $\mathbb{B}_R^n$ the closed ball in an $n$--dimensional Euclidean space centered at the origin with radius $R$. We use $|\bx|$ to denote the Euclidean norm of a vector $\bx\in\RR^n$. The distance between a point $\bx$ and a set $\mathcal{A}$ in a Euclidean space is defined as $\bd(\bx,\mathcal{A}) = \inf\{|\by-\bx| : \by\in \mathcal{A}\}$, and the directed Hausdorff distance from a set $\mathcal{A}$ to a set $\mathcal{B}$ is defined by $\bd(\mathcal{A},\mathcal{B}) = \sup\{\bd(\bx,\mathcal{B}) : \bx \in \mathcal{A}\}$. For any integer $m\in\NN$, we denote $[m]:=\{1,\ldots,m\}$.

We use $\bx=(x_1,\dots,x_n)$ to denote a vector of variables and $\RR[\bx]$ to denote the ring of polynomials in $\bx$. Let $\alpha=(\alpha_1,\dots,\alpha_n)$ be a multi-index with nonnegative integer components and length $|\alpha|=\sum_{i=1}^n \alpha_i$. The set of multi-indices of length at most $r$ is denoted by $\NN^n_r=\{\alpha\in\NN^n:|\alpha|\leq r\}$, and the number of monomials of degree at most $r$ is denoted by $s(n,r)=\binom{n+r}{r}$. We let $\overline{\NN}^n_r$ denote the subset of $\NN^n_r$ whose elements have length exactly $r$. The monomials in $\bx$ are written in the form $\bx^\alpha=x_1^{\alpha_1}\cdots x_n^{\alpha_n}$. For any polynomial $f(\bx)=\sum_{\alpha}f_{\alpha}\bx^{\alpha}\in\RR[\bx]$, we define the norms $\monc{f}=\sqrt{\sum_{\alpha}f_{\alpha}^2}$, $\mon{f} = \sum_{\alpha \in \NN^n}|f_{\alpha}|$, and $\lceil f\rceil:=\lceil\deg(f)/2\rceil$. Let $\langle \cdot, \cdot \rangle$ denote the inner product between vectors, and $\vc(f)$ denotes the vector of coefficients of a polynomial $f$. Denote the Kronecker product by $\otimes$. For a vector, $\|\cdot\|$ denotes the Euclidean norm. For a polynomial, we always specify the coefficient norm by $\mon{\cdot}$ or $\monc{\cdot}$. We write $\cheb{f}$ for the sum of the absolute values of the coefficients of $f$ in the tensor Chebyshev basis $T_\alpha(\bx)=\prod_iT_{\alpha_i}(x_i)$ on $[-1,1]^n$. On $\mathrm{H}^n$, we use $\cheb{f\circ\mathbf A}$, where $\mathbf A(\bx)=(\bx+\mathbf1)/2$. The comparisons between these norms are proved in Appendix~\ref{app:norms}. We use $\log_2$ for the logarithm to base two.

For a polynomial $f$ and a nonempty compact set $K$, we write $\supnorm fK:=\max_{\bx\in K}|f(\bx)|$ for the uniform norm of $f$ on $K$.

We use $\mathrm H^n=[0,1]^n$ throughout the scalar sections, and use $[-1,1]^n$ in Section~\ref{sec:matrix}. The unit ball and $\mathrm H^n$ are defined by 
\begin{displaymath}
   \mathrm{B}^n =\{\bx \in \RR^n\;:\; 1 - \|\bx\|^2 \geq 0\},\quad \mbox{and} \quad \mathrm{H}^n = \{\bx \in \RR^n\;:\; x_i(1 - x_i) \geq 0 \; \forall \; i \in [n]\}. 
\end{displaymath}

\subsection{Moment sequence}

From the viewpoint of the moment-SOS hierarchy, the formulations~\eqref{hierarchy: poly-preordering},~\eqref{hierarchy: poly-quadratic},~\eqref{hierarchy: LP}, and~\eqref{hierarchy: Hol and Scherer} are referred to as the dual problems, whose primal problems are conic programs over pseudo-moment sequences. In this section, we briefly recall this concept. 

Write $\bv(\bx)=(\bx^\alpha)_{\alpha\in\NN^n}$ and $\bv_k(\bx)=(\bx^\alpha)_{|\alpha|\leq k}$. Let $\biy$ be an infinite-dimensional vector indexed by $\bv(\bx)$. $\biy:= (\iy_{\alpha})_{\alpha \in \NN^n}$ is said to be a {\em moment sequence} on $\cX$ if there exists a probability measure $\mu$ supported on $\cX$ such that for any multi-index $\alpha \in \NN^n$, the $\alpha$--component of $\biy$ is equal to the $\alpha$--moment of $\mu$, equivalently, 
\begin{displaymath}
    \int_{\cX} \bx^\alpha~d\mu = \iy_{\alpha} \quad \forall \alpha \in \NN^n.
\end{displaymath}
Let $\CM(\cX)$ denote the set of all moment sequences supported on $\cX$. For any positive integer $r$, we define the $r$--truncated moment sequence $\biy=(\iy_{\alpha})_{\alpha \in \NN^n_r}$ to be obtained by taking all components of a moment sequence indexed by multi-indices of length at most $r$. We denote this set by $\CM_r(\cX)$. It is convex, and can be approximated by the set of {\em pseudo-moment sequences}, which can be verified by either an SDP or an LP. 

Because of Tchakaloff's theorem (see e.g.,~\cite{Tchakerloffs_theorem}), for any measure $\mu$ supported on a compact set $\cX$, there exist at most $s(n,k)$ points $\{\bx_j: j\in[s(n,k)]\} \subset \cX$, and corresponding positive weights $\{w_j: j\in [s(n,k)]\}$ satisfying $\sum_{j=1}^{s(n,k)} w_j=1$ such that 
for any $f \in \RR[\bx]_k$, the integral of $f$ over $\CX$ with respect to $\mu$ can be calculated as follows: 
\begin{displaymath}
    \int_{\CX}f(\bx) d\mu(\bx) = \sum_{j=1}^{s(n,k)}w_j f(\bx_j).
\end{displaymath}
This implies that when considering an element $y\in \CM_k(\CX)$, its corresponding measure can always be assumed to be a discrete measure with support contained in $\CX$, equivalently
\begin{displaymath}
        \biy = \sum_{i=1}^{s(n,k)} w_i\bv_k(\bx_i),\qquad \bx_i \in \cX \; \forall i \in [s(n,k)].
\end{displaymath}

Let $\biy = (\iy_{\alpha})_{\alpha \in \NN^n}$ be a real sequence indexed by the
vector of monomials $\bv(\bx)$. We define the Riesz linear functional $\ell_{\biy}: \RR[\bx] \to \RR$ as follows: 
\begin{displaymath}
    f(\bx)= \sum_{\alpha \in \NN^n}f_{\alpha}\bx^{\alpha} \quad \mapsto \quad \ell_{\biy}(f) = \sum_{\alpha \in \NN^n}f_{\alpha}\iy_{\alpha} = \langle \vc(f),\biy \rangle,
\end{displaymath}
For a finite sequence $(\iy_\alpha)_{|\alpha|\leq q}$, the same formula defines $\ell_{\biy}$ only on $\RR[\bx]_q$. Thus a full sequence at relaxation level $r$ defines a functional on $\RR[\bx]_{2r}$ for the SOS hierarchies and on $\RR[\bx]_r$ for the preprime hierarchies.

The Riesz linear functional plays a central role in determining whether a sequence $\biy$ is a moment sequence for a Borel measure (see the Riesz-Haviland Theorem, e.g., \cite[Theorem 3.1]{lasserre2009moments}). As a result, we can use $\ell_{\biy}$ to define the moment matrix and localizing matrix as follows: let $\cX$ be the semi-algebraic set defined as in~\eqref{def: poly_set}, for such $\biy$ indexed in $\bv(\bx)$, the moment matrix $\bM(\biy)$ with rows and columns indexed by $\bv(\bx)$ is defined by
\begin{displaymath}
    \bM(\biy)(\alpha, \beta) = \ell_{\biy}(\bx^{\alpha+\beta})=\iy_{\alpha +\beta}, \quad \forall \alpha, \beta \in \NN^n.
\end{displaymath}
For a given $r \in \NN$, the $r$-truncated moment matrix, denoted by $\bM_r(\biy)$, is the submatrix of $\bM(\biy)$ obtained by extracting the rows and columns of $\bM(\biy)$
indexed by $\bv_r(\bx)$. Its entries use moments through degree $2r$; the index $r$ is the degree of its row and column monomials.
Similarly, for a given polynomial $g \in \RR[\bx]$, the localizing matrix $\bM(g\biy)$ associated with $\biy$ and $g$ is defined by 
\begin{displaymath}
    \bM(g\biy)(\alpha,\beta)=\ell_{\biy}(g(\bx)\bx^{\alpha+\beta})
    = \sum_{\gamma} g_\gamma \iy_{\gamma+\alpha+\beta}, \quad \forall \alpha, \beta \in \NN^n.
\end{displaymath}
The $r$-truncated localizing matrix is similarly constructed by extracting all the rows and columns indexed by $\bv_r(\bx)$ from the localizing matrix $\bM(g\biy)$. For a sequence through degree $2r$, the localizer used in the hierarchy is $\bM_{r-\lceil g\rceil}(g\biy)$, since its entries have degree at most $2(r-\lceil g\rceil)+\deg g\leq2r$. 

In the matrix case, where $\cX$ is defined as in~\eqref{def: matrix_set}, we define the localizing matrix $\bM_r(G\biy)$ by
    \begin{displaymath}
        \bM_r(G\biy) \;=\; \ell_{\biy} \big(G(\bx) \otimes (\bv_r(\bx)\bv_r(\bx)^{\top})\big),
    \end{displaymath}
    where we slightly abuse the notation of the Riesz functional to mean that $\ell_{\biy}$ acts entry-wise on the polynomial matrix $G(\bx) \otimes (\bv_r(\bx)\bv_r(\bx)^{\top})$. Consequently, the work~\cite{tran2025convergence,tran2024convergence,lasserre2005polynomial} presented the primal formulation for the hierarchies~\eqref{hierarchy: poly-preordering},~\eqref{hierarchy: poly-quadratic},~\eqref{hierarchy: LP}, and~\eqref{hierarchy: Hol and Scherer} as follows: 
    \begin{enumerate}
        \item The primal formulation for the Schm\"udgen-type hierarchy for $\cX$ defined in~\eqref{def: poly_set} is defined by
        \begin{align}\label{hierarchy: moment Schmudgen}
        & \mlb(f,\CT(\CX))_r = \inf \Big\{ \ell_{\biy}(f) = \sum_{\alpha \in \NN^n_{2r}}f_{\alpha}y_{\alpha} \,:\, \biy \in \pMTX2r\Big\}
        \\
        \text{where \quad} & \pMTX2r := \Big\{ 
        \biy\in 
        \RR^{s(n,2r)}\,:\,\ y_0 =1,\ \bM_r(\biy) \succeq 0,\nonumber\\ &\hspace{20mm}\ell_{\biy}(\bx^\alpha h_i)=0\; (|\alpha|+\deg h_i\leq2r)\;\forall i\in[p],\nonumber\\ 
        &\hspace{30mm}  \bM_{r-\lceil g_J \rceil}(g_J\biy) \succeq 0\;\;
        \forall\; J \subset [m] \;\mbox{such that} \; \lceil g_J \rceil\leq r \Big\}
        \nonumber.
\end{align}
        \item The primal formulation for the Putinar-type hierarchy for $\cX$ defined in~\eqref{def: poly_set} is defined by
        \begin{align}\label{hierarchy: moment Putinar}
        & \mlb(f,\CQ(\CX))_r = \inf\Big\{ \ell_{\biy}(f) = \sum_{\alpha \in \NN^n_{2r}}f_{\alpha}y_{\alpha}\,:\, \biy \in \pMQX2r \Big\}
        \\
        \text{where \quad}& \pMQX2r :=\Big\{ \biy \in \RR^{s(n,2r)}\,:\, y_0 =1,\ \bM_r(\biy) \succeq 0,\nonumber \\
         &\hspace{50mm}  \bM_{r-\lceil g_i \rceil}(g_i\biy) \succeq 0 \ \forall\; i \in [m],\nonumber\\ &\hspace{35mm}\ell_{\biy}(\bx^\alpha h_i)=0\; (|\alpha|+\deg h_i\leq2r)\;\forall i\in[p]\Big\}.
        \nonumber
        \end{align}
        \item The primal formulation for the extended-Handelman-type hierarchy for $\cX$ defined in~\eqref{def: poly_set}, where the certificate of positivity $\CR(\cX)$ is defined as in~\eqref{def: extended preprime}, is defined by
        \begin{align}\label{hierarchy: moment Handelman}
        & \mlb(f,\CR(\CX))_r = \inf\Big\{ \ell_{\biy}(f) = \sum_{\alpha \in \NN^n_{r}}f_{\alpha}y_{\alpha}\,:\, \biy \in \pMRX2r \Big\}
        \\
        \text{where \quad}& \pMRX2r :=\Big\{ \biy \in \RR^{s(n,r)}\,:\, y_0 =1,\ \ell_{\biy}(\bx^{\gamma_i} h_i) =0\ \forall \gamma_i \in \NN^n_{r-\deg(h_i)},\; i \in [p]  \nonumber \\
         & \ell_{\biy}(\bg^{\alpha}\bx^\beta(1-\bx)^\gamma) \geq 0,\forall \alpha \in \NN^m,\ \beta, \gamma\in \NN^n,\; \deg(\bg^{\alpha})+|\beta|+|\gamma|\leq r\Big\}.
        \nonumber
        \end{align}
        \item The primal formulation for the Hol-Scherer-type hierarchy for $\cX$ defined in~\eqref{def: matrix_set} is defined by
        \begin{align}\label{hierarchy: moment Hol-Scherer}
        & \mlb(f,\CQ(\CX))_r = \inf\Big\{ \ell_{\biy}(f) = \sum_{\alpha \in \NN^n_{2r}}f_{\alpha}y_{\alpha}\,:\, \biy \in \pMQX2r \Big\}
        \\
        \text{where \quad}& \pMQX2r :=\Big\{ \biy \in \RR^{s(n,2r)}\,:\, y_0 =1,\ \bM_r(\biy) \succeq 0,\; \bM_{r-\lceil G \rceil}(G\biy) \succeq 0\Big\},
        \nonumber
        \end{align}
        where $\degg{G}$ denotes $\max_{i,j \in [m]}\degg{g_{ij}}$, and recall that $G(\bx) = (g_{ij})_{m \times m}$ as defined in~\eqref{def: matrix_set}. 
    \end{enumerate}

    \subsection{Hausdorff distances}\label{sec: Hausdorff distance}
    Consider the problem~\eqref{POP}, where the feasible set $\cX$ is defined either in~\eqref{def: poly_set} or~\eqref{def: matrix_set}. In this paper, we denote the degree of $f$ by $k$. Then the problem~\eqref{POP} admits a reformulation in terms of $k$--truncated moment sequences:
    \begin{equation*}
     \fmin = \inf\left\{\int_{\CX}fd\mu: \mu \in \CP(\CX) \right\}=\inf\left\{\sum_{\alpha \in \NN^n_k} f_{\alpha}y_{\alpha}=\langle \vc(f),\biy \rangle:\ \biy \in \CM_k(\CX) \right\},
\end{equation*}
where $\CM_k(\cX)$ denotes the set of $k$--truncated moment sequences on $\cX$. However, verifying membership of $\CM_k(\cX)$ is challenging, so we approximate it by the set of truncated pseudo-moment sequences. In particular, the primal formulations~\eqref{hierarchy: moment Schmudgen},~\eqref{hierarchy: moment Putinar},~\eqref{hierarchy: moment Handelman}, and~\eqref{hierarchy: moment Hol-Scherer} admit the following reformulations: 
\begin{eqnarray*}
    \mlb(f,\CT(\CX))_r = \inf\left\{\sum_{\alpha \in \NN^n_k} f_{\alpha}y_{\alpha}=\langle \vc(f),\biy \rangle:\ \biy \in \pMkTX2r \right\},
    \\
    \mlb(f,\CQ(\CX))_r = \inf\left\{\sum_{\alpha \in \NN^n_k} f_{\alpha}y_{\alpha}=\langle \vc(f),\biy \rangle:\ \biy \in \pMkQX2r \right\},
    \\
    \mlb(f,\CR(\CX))_r = \inf\left\{\sum_{\alpha \in \NN^n_k} f_{\alpha}y_{\alpha}=\langle \vc(f),\biy \rangle:\ \biy \in \pMkRX2r \right\}.
\end{eqnarray*}
Here, for $k\leq r$, let $\pi_k^r :\RR^{s(n,r)} \to \RR^{s(n,k)}$ denote the projection onto the first $s(n,k)$ coordinates, then 
\begin{eqnarray*}
\pMkTX2r &=& \{ \pi_k^{2r}(\biy)\in \RR^{s(n,k)} \,:\, \biy\in \pMTX2r \},
\\
\pMkQX2r &=& \{ \pi_k^{2r}(\biy)\in \RR^{s(n,k)} \,:\, \biy\in \pMQX2r \},
\\
\pMkRX2r &=& \{ \pi_k^r(\biy)\in \RR^{s(n,k)} \,:\, \biy\in \pMRX2r \}.
\end{eqnarray*}
Note that when $\cX$ is defined as in~\eqref{def: poly_set} or~\eqref{def: matrix_set}, the truncated quadratic modules are both defined by $\CQ(\cX)$; they will be studied in separate sections to avoid confusion. 
To analyze the tightness of these approximations, we define the following Hausdorff distances:
\begin{eqnarray*}
\quad \bd_k(\CT(\CX)_{2r}) &:=& \bd(\pMkTX2r,\CM_k(\CX)) =\sup\{\bd(\biy,\CM_k(\CX)):\ \biy \in \pMkTX2r\},\\
\quad \bd_k(\CQ(\CX)_{2r}) &:=& \bd(\pMkQX2r,\CM_k(\CX))=\sup\{\bd(\biy,\CM_k(\CX)):\ \biy \in \pMkQX2r\},\\
\quad \bd_k(\CR(\CX)_{r}) &:=& \bd(\pMkRX2r,\CM_k(\CX))=\sup\{\bd(\biy,\CM_k(\CX)):\ \biy \in \pMkRX2r\}.
\end{eqnarray*}
Under the coordinate-generator assumptions used below, the pseudo-moment sets are compact, so these suprema are attained. We use these Hausdorff distances to bound the gap between the lower bounds given by the primal formulations and the global optimal value $\fmin$ as presented in the following lemma. 
\begin{lemma}\label{lemma: distance to convergence rate}
    Let $\CX$ be a compact semi-algebraic set and $k \geq {\rm deg}(f)$. Assume $k\leq2r$ for the SOS hierarchies and $k\leq r$ for the preprime hierarchy. Then the gaps of the optimal value $f_{\min}$ and the optimal value $f_r$ of the $r$-level of the moment hierarchies are bounded proportionally to the Hausdorff distances as follows:
    
    \begin{eqnarray*}
        \qquad \fmin - \mlb(f,\CT(\CX))_r &\leq& \monc{f}\,\bd_k(\CT(\CX)_{2r}),\\
        \qquad \fmin - \mlb(f,\CQ(\CX))_r &\leq&\monc{f}\,\bd_k(\CQ(\CX)_{2r}),\\
        \qquad \fmin - \mlb(f,\CR(\CX))_r &\leq&\monc{f}\,\bd_k(\CR(\CX)_{r}).
    \end{eqnarray*}
\end{lemma}
\begin{proof}
    This lemma is a direct consequence of the Cauchy--Schwarz inequality; we refer to~\cite[Lemma~2.4]{tran2025convergence} for details.
\end{proof}

\subsection{{\L}ojasiewicz inequality}
The {\L}ojasiewicz inequality plays an important role in our method. In this section, we present two versions of the {\L}ojasiewicz inequality associated to different descriptions of $\cX$ as in~\eqref{def: poly_set} and~\eqref{def: matrix_set}. 
\begin{lemma}[\cite{Real_algebraic_geometry},Corollary 2.6.7]\label{lemma: Lojasiewicz inequality}
     Let $B$ be a compact semi-algebraic set, and $f$ and $g$ be two continuous 
     semi-algebraic functions from $B$ to $\RR$ such that $f^{-1}(0) \subset g^{-1}(0)$. Then there exist a {\L}ojasiewicz constant $c > 0 $ and a {\L}ojasiewicz exponent $1 \geq \Lo > 0$ such that 
     \begin{displaymath}
         |g(\bx)| \leq c|f(\bx)|^{\Lo} \quad \forall\; \bx \in B.
     \end{displaymath}
 \end{lemma}
 In the scalar setting, where $\cX$ is defined in~\eqref{def: poly_set}, we define the violating function
 \begin{displaymath}
     \bp(\bx) := \max\{0,\;-g_j(\bx),\;|h_i(\bx)|\ :\ j\in [m],\; i \in [p] \}, 
 \end{displaymath}
 which can be seen as a function to define $\cX$ as $\cX = \bp^{-1}(0):=\{\bx \in \RR^n:\ \bp(\bx) =0\}$. Since $\bp$ and $\bd(\bx,\cX)$ are two continuous semi-algebraic functions (see e.g.,~\cite{Real_algebraic_geometry}), for any compact semi-algebraic set $B$ containing $\cX$, there exists the {\L}ojasiewicz constant $c_{\Lo}$ and the {\L}ojasiewicz exponent $\Lo$ such that 
 \begin{equation}\label{eq: poly_Lojasiewicz inequality}
     \bd(\bx,\cX) \leq c_{\Lo} \cdot \bp(\bx)^{\Lo}\quad \forall \ \bx \in B.
 \end{equation}

 In the matrix setting, where $\cX$ is defined as in~\eqref{def: matrix_set}, another version of the {\L}ojasiewicz inequality was studied in~\cite{dinh2016lojasiewicz}, which is presented in the following lemma.

  \begin{lemma}\cite[Theorem 4.1]{dinh2016lojasiewicz}
 \label{lem: matrix Lojasiewicz}
For any compact set $B$ containing $\cX$ as defined  in \eqref{def: matrix_set}, there exist a {\L}ojasiewicz constant $c_{\Lo}>0$ and a {\L}ojasiewicz exponent $0 < \Lo \leq 1$ depending on $B$ and the defining matrix $G$ such that
\begin{displaymath}
    \bd(\bx,\mathcal{X})\leq c_{\Lo}\cdot \max\{0,-\lambda_m(\bx)\}^{\Lo} \quad \forall\  \bx \in B,
\end{displaymath}
where $\lambda_1(\bx) \geq \ldots \geq \lambda_m(\bx)$ denote the eigenvalue functions of $G(\bx)$ in the non-increasing order.
\end{lemma}

Throughout the paper, the exponent $\Lo$ is used in the form of~\eqref{eq: poly_Lojasiewicz inequality}. By decreasing it on a compact set if necessary, we take $0<\Lo\leq1$. The scalar and spectral violation functions satisfy the hypotheses of the common atom estimate~\eqref{eq: atom transfer}, proved in Stage 2 of Section~\ref{sec: method}. We use~\eqref{eq: universal moment transfer} after estimating the moment displacement and the average squared violation in each hierarchy.

\section{Putinar-type moment-SOS hierarchy}\label{sec: Put}
In this section, we aim to explore the convergence rate of the Putinar-type moment-SOS hierarchy, namely the hierarchy~\eqref{hierarchy: poly-quadratic}, for the problem~\eqref{POP} with $\cX$ defined as in~\eqref{def: poly_set}. We make the following Archimedean assumption, which is standard in studying convergence rate of the Putinar-type hierarchy. 
\begin{assume}\label{assume: hypercube containment}
    Assume that $\cX$ is nonempty and contained in the hypercube $\mathrm{H}^n$, and the description of $\cX$ is given by 
    \begin{displaymath}        
    \cX := \{ \bx \in \RR^n:\; x_i(1-x_i) \geq 0 \; \forall i \in [n],\; g_j(\bx) \geq 0 \; \forall j \in [m],\; h_i(\bx) =0\; \forall i \in [p]\}. 
    \end{displaymath}
    Then the Archimedean condition is satisfied for $\cX$. 
\end{assume}
We next follow the two stages described in Section~\ref{sec: method}.
\subsection{Stage 1: Base estimation} We collect the results from existing works to prove the strong duality, compactness, and the upper bound on the Hausdorff distance $\bd_k(\CQ(\cX)_{2r})$.

We first establish compactness and strong duality. The coordinate certificates used below also replace the general ball-radius estimate by a sharp bound for the hypercube.
For every multi-index $\alpha$,
\begin{equation}\label{eq: coordinate certificate}
 1\pm\bx^\alpha\in\CQ(\mathrm{H}^n)_{2\lceil|\alpha|/2\rceil}.
\end{equation}
Indeed, $1-x_i^2=(1-x_i)^2+2x_i(1-x_i)\in\CQ(\mathrm{H}^n)_2$. For any $\beta\in\NN^n$, the identity
\[
 1-\bx^{2\beta}=\sum_{i=1}^n\left(\prod_{j<i}x_j^{2\beta_j}\right)
       \left(\sum_{a=0}^{\beta_i-1}x_i^{2a}\right)(1-x_i^2)
\]
gives a certificate of degree at most $2|\beta|$, where an empty sum is zero. Choose $\beta,\gamma\in\NN^n$ with $\alpha=\beta+\gamma$ and $|\beta|,|\gamma|\leq\lceil|\alpha|/2\rceil$. Then
\[
 1\pm\bx^\alpha=\tfrac12\big((\bx^\beta\pm\bx^\gamma)^2
                    +(1-\bx^{2\beta})+(1-\bx^{2\gamma})\big),
\]
which proves~\eqref{eq: coordinate certificate}. Consequently, for any $\biy\in\CP\CM(\CQ(\mathrm{H}^n)_{2r})$,
\begin{equation}\label{eq: sharp radius}
 |\iy_\alpha|\leq1\quad(|\alpha|\leq2r),\qquad
 \|\pi_k^{2r}(\biy)\|\leq\sqrt{s(n,k)}\quad(k\leq2r).
\end{equation}
The radius is attained by the moment sequence of the Dirac measure at $(1,\ldots,1)$. Since the full pseudo-moment set is closed and bounded, it is compact; its coordinate projections are compact as well. The same conclusion holds when further scalar or matrix constraints are imposed.

The certificates~\eqref{eq: coordinate certificate} also show that $1$ is an interior point of the truncated cone in the space of polynomials of degree at most $2r$. To see this, any polynomial $p$ satisfies
\begin{equation}\label{eq: coefficient certificate}
 \mon{p}\pm p\in\CQ(\mathrm{H}^n)_{2\lceil\deg(p)/2\rceil}.
\end{equation}
Thus a coefficient neighborhood of $1$ is contained in the cone. Finite-dimensional conic duality then gives
$\lb(f,\CQ(\cX))_r=\mlb(f,\CQ(\cX))_r$ whenever $2r\geq\deg f$ and the relevant constraints are included. More explicitly, separation identifies the infimum over normalized dual functionals with the supremum of $t$ for which $f-t$ belongs to the closure of the cone. Subtracting any $\epsilon>0$ from such a bound adds $\epsilon\cdot1$, an interior element, and gives membership in the cone itself. Hence taking its closure does not change the supremum. This argument also applies to the matrix quadratic module containing the constraints $x_i(1-x_i)$.

\begin{theorem}[\cite{gribling2026squared}, Theorem 7]\label{thm: rate on box}
Consider a polynomial $f$ of degree at most $d\geq1$, and let $f_{\min}=\min_{\bx\in[-1,1]^n}f(\bx)$. There are absolute positive constants $c_1\leq70458$ and $c_2\leq540$ such that, for every integer $r\geq\max\{2,d\}$ satisfying $n\mid r$ and $r/\log_2r\geq c_2n^2d^2$,
\[
 f_{\min}-\lb(f,\CQ([-1,1]^n))_r
 \leq c_1n^3d^2\cheb{f}\frac{\log_2^3r}{r^2}.
\]
\end{theorem}
Here $\CQ([-1,1]^n)$ is generated by $1-x_i^2$, $i\in[n]$. We next use an affine polynomial transformation to obtain a version of Theorem~\ref{thm: rate on box} for the hypercube $\mathrm{H}^n=[0,1]^n$.
\begin{theorem}\label{thm: rate on hypercube}
Under the same conditions on $n,d,r$, let $f_{\min}=\min_{\bx\in\mathrm{H}^n}f(\bx)$ and $\mathbf A(\bx)=(\bx+\mathbf1)/2$. Then
\begin{align*}
 f_{\min}-\lb(f,\CQ(\mathrm{H}^n))_r
 &\leq c_1n^3d^2\cheb{f\circ\mathbf A}\frac{\log_2^3r}{r^2}\\
 &\leq c_1n^3d^2\sqrt{s(n,d)}\monc{f}\frac{\log_2^3r}{r^2}.
\end{align*}
\end{theorem}
\begin{proof}
The map $\mathbf A$ sends $[-1,1]^n$ onto $\mathrm{H}^n$ and sends $x_i(1-x_i)$ to $(1-x_i^2)/4$. It therefore identifies the corresponding truncated quadratic modules without changing their degrees. Apply Theorem~\ref{thm: rate on box} to $f\circ\mathbf A$ and use
\[
 \cheb{f\circ\mathbf A}\leq\mon{f}\leq\sqrt{s(n,d)}\monc{f},
\]
proved in Appendix~\ref{app:norms}.
\end{proof}

\begin{theorem}\label{thm: Hausdorff over H^n}
For positive integers $k,n,r$ satisfying $r\geq\max\{2,k\}$, $n\mid r$, and $r/\log_2r\geq c_2n^2k^2$, the Hausdorff distance over the hypercube admits the upper bound
\begin{equation}\label{eq: cube Hausdorff bound}
 \bd_k(\CQ(\mathrm{H}^n)_{2r})
 \leq c_1n^3k^2\sqrt{s(n,k)-1}\frac{\log_2^3r}{r^2}.
\end{equation}
For every $f\in\RR[\bx]_k$, with $f_{\min}=\min_{\bx\in\mathrm H^n}f(\bx)$, the corresponding optimal-value bound is
\[
 0\leq f_{\min}-\lb(f,\CQ(\mathrm H^n))_r
 \leq\monc f\,c_1n^3k^2\sqrt{s(n,k)-1}\frac{\log_2^3r}{r^2}.
\]
In particular, for fixed $n,k$ and $f$, both the Hausdorff distance and the optimal-value gap converge at rate $\mathrm O((\log_2r)^3/r^2)$.
\end{theorem}
\begin{proof}
By~\eqref{eq: sharp radius}, the set $\CP\CM_k(\CQ(\mathrm{H}^n)_{2r})$ is compact. The set $\CM_k(\mathrm{H}^n)$ is also compact and convex, by its atomic representation. Therefore there are $\overline{\biy}\in\CP\CM_k(\CQ(\mathrm{H}^n)_{2r})$ and its unique projection $\widetilde{\biy}\in\CM_k(\mathrm{H}^n)$ such that
\[
 \bd_k(\CQ(\mathrm{H}^n)_{2r})=\|\widetilde{\biy}-\overline{\biy}\|.
\]
The first-order optimality condition gives
\[
 \langle\widetilde{\biy}-\overline{\biy},\biy-\widetilde{\biy}\rangle\geq0
 \qquad\forall\biy\in\CM_k(\mathrm{H}^n).
\]
Consider the polynomial optimization problem
\begin{equation*}
 \min_{\bx\in\mathrm{H}^n}f(\bx),\qquad
 f(\bx)=\langle\widetilde{\biy}-\overline{\biy},\bv_k(\bx)\rangle.
\end{equation*}
Its minimum is $\ell_{\widetilde{\biy}}(f)$, whereas its moment relaxation is at most $\ell_{\overline{\biy}}(f)$. Consequently,
\[
 \bd_k(\CQ(\mathrm{H}^n)_{2r})^2
 \leq f_{\min}-\mlb(f,\CQ(\mathrm{H}^n))_r
 \leq c_1n^3k^2\cheb{f\circ\mathbf A}\frac{\log_2^3r}{r^2}.
\]
Both sequences have constant component $1$, so $f$ has zero constant coefficient. Appendix~\ref{app:norms} gives
\[
 \cheb{f\circ\mathbf A}\leq\mon{f}
 \leq\sqrt{s(n,k)-1}\|\widetilde{\biy}-\overline{\biy}\|.
\]
Dividing by this distance when it is nonzero proves~\eqref{eq: cube Hausdorff bound}; the zero case is immediate. A constant shift of the auxiliary objective is unnecessary, which also removes the radius factor from the final estimate.

The optimal-value bound follows from Lemma~\ref{lemma: distance to convergence rate} and strong duality. The asymptotic rates hold for all sufficiently large orders by monotonicity, using the largest multiple of $n$ not exceeding $r$.
\end{proof}

\subsection{Stage 2: Lift and projection}\label{sec: Put lift and project}
We construct a lift of $\cX$ and of a full feasible extension of each sequence in $\pMkQX2r$. Recall that $\cX$ is defined by 
\begin{displaymath}
    \cX = \{\bx \in \RR^n\;:\; x_i(1-x_i) \geq 0 \; \forall i \in [n],\; g_j(\bx) \geq 0 \ \forall \ j \in [m],\; h_i(\bx) = 0 \ \forall \ i \in [p] \}. 
\end{displaymath}
The lifted set $\cK$ in Stage 2 is defined as 
\begin{displaymath}
    \cK := \{\bz = (\bx,\bu) \in \RR^{n+m}\;:\; u_j \geq 0,\; \bx \in \cX,\; \; \widehat{g}_j(\bz):=g_j(\bx)-u_j =0\; \forall \; j \in [m]\}.
\end{displaymath}
Since for any $\bx \in \cX \subset \mathrm{H}^n$, $g_j(\bx) \leq \sum_{\alpha \in \NN^n} |g_{j,\alpha}| = \mon{g_j}$. Without loss of generality, we can assume $\mon{g_j} \leq 1$ for all $j \in [m]$ so that $\cK$ is contained in the box $\mathrm{H}^{n+m}$.

We set $d:=\max\{1,\deg g_1,\ldots,\deg g_m\}$. We next define the substitution map on a sequence through degree $k$. For any $\biy\in\RR^{s(n,k)}$ and integer $t\leq\lfloor k/d\rfloor$, let
\begin{displaymath}
    \varphi_k^t:\; \RR^{s(n,k)} \to \RR^{s(n+m,t)},\quad \biy \mapsto \biy^{k,t},
\end{displaymath}
where we use conventions $\bz^{(\alpha,\beta)} = \bx^\alpha\bu^\beta$ for $\alpha \in \NN^n$ and $\beta \in \NN^m$, and $\bg = (g_1,\dots,g_m)$ to define $\biy^{k,t}$ by value of the Riesz functional on monomials in $\bz$ as
\begin{equation*}
\ell_{\biy^{k,t}}(\bz^{(\alpha,\beta)}) \;=\; \ell_{\biy}(\bx^\alpha\bg(\bx)^\beta).
\end{equation*}
This map is well-defined since for any $\alpha \in \NN^n$ and $\beta \in \NN^m$ such that $|\alpha| + |\beta| \leq t$, the degree of $\bx^\alpha \bg(\bx)^{\beta}$ admits the upper bound
\begin{displaymath}
    \deg(\bx^\alpha \bg(\bx)^{\beta}) \leq d\cdot t \leq k.
\end{displaymath}
This mapping $\varphi_k^t$ possesses the three key properties that play central roles in our later estimations, which are presented in the following lemmas. 
\begin{lemma}\label{lem: Putinar moment sequence}
    Let $k$ be a positive integer and $t \leq \lfloor k / d \rfloor$, if $\biy \in \CM_k(\cX)$ then $\biy^{k,t} \in \CM_t(\cK)$.
\end{lemma}
\begin{proof}
    Let $\biy \in \CM_k(\cX)$. Since $\cX$ is compact, Tchakaloff's theorem indicates that $\biy$ admits an atomic representation as 
    \begin{displaymath}
        \biy = \sum_{s=1}^{s(n,k)} w_s\bv_k(\bx_s),\qquad \bx_s \in \cX \; \forall s \in [s(n,k)], \quad w_s \geq 0 \ \forall \ s \in [s(n,k)],\; \sum_{s=1}^{s(n,k)} w_s =1. 
\end{displaymath}
Since $\varphi$ is a linear transformation, and the image of $\bv_k(\bx)$ via $\varphi$ can be viewed as 
\begin{displaymath}
    (\bv_k(\bx))^{k,t} = \bv_t(\bx,\bg(\bx)),\quad  \mbox{and}\quad (\bx,\bg(\bx)) \in \cK \quad \forall \bx \in \CX,
\end{displaymath}
we claim that 
\begin{displaymath}
    \biy^{k,t} = \sum_{s=1}^{s(n,k)} w_s\bv_t(\bx_s,\bg(\bx_s)) \in \CM_t(\cK).
\end{displaymath}
\end{proof}

\begin{lemma}\label{lem: Putinar pseudo moment sequence}
Let $r^\prime=\lfloor r/d\rfloor$. For any $\biy\in\CP\CM(\CQ(\cX)_{2r})$ and $t\leq2r^\prime$, the substitution lift $\biy^{2r,t}$ is a truncated pseudo-moment sequence on $\cK$ described by the original constraints in $\bx$, the inequalities $u_j\geq0$, and the equations $g_j-u_j=0$.
\end{lemma}
\begin{proof}
It suffices to consider $t=2r^\prime$. For any polynomial $p(\bx,\bu)$ of degree at most $r^\prime$, its substitution $p(\bx,\bg(\bx))$ has degree at most $dr^\prime\leq r$. Hence
\[
 \ell_{\biy^{2r,2r^\prime}}(p^2)
 =\ell_{\biy}(p(\bx,\bg(\bx))^2)\geq0.
\]
For $\deg p\leq r^\prime-1$, the same argument gives
\[
 \ell_{\biy^{2r,2r^\prime}}(u_jp^2)
 =\ell_{\biy}(g_jp(\bx,\bg(\bx))^2)\geq0,
\]
because $2d(r^\prime-1)+\deg g_j\leq2r$. Localizers for the original constraints are treated in the same way. Normalization is preserved. Every multiple of $g_j-u_j$ vanishes after substitution, and every allowed multiple of $h_i$ is mapped to a multiple of $h_i$ of degree at most $2r$. This proves the linear conditions as well.
\end{proof}

\begin{lemma}\label{lem: Putinar pseudo moment sequence for box}
Let $r^\prime$ be a positive integer satisfying $d(r^\prime+2)\leq r$. For any $\biy\in\CP\CM(\CQ(\cX)_{2r})$ and $t\leq2r^\prime$,
\[
 \biy^{2r,t}\in\CP\CM_t(\CQ(\mathrm{H}^{n+m})_{2r^\prime}).
\]
Moreover, its Riesz functional annihilates $h_i^2$ and $(g_j-u_j)^2$ whenever their degrees are at most $t$.
\end{lemma}
\begin{proof}
The moment matrix, normalization and localizing conditions for $x_i(1-x_i)$ follow as in Lemma~\ref{lem: Putinar pseudo moment sequence}. It remains to establish the localizers for $u_j(1-u_j)$. Since $\mon{g_j}\leq1$,~\eqref{eq: coefficient certificate} implies
\[
 1-g_j\in\CQ(\mathrm{H}^n)_{2\lceil d/2\rceil}.
\]
The identity
\[
 g_j(1-g_j)=g_j^2(1-g_j)+(1-g_j)^2g_j
\]
therefore provides a certificate in $\CQ(\cX)_{4d}$. For every $p\in\RR[\bx,\bu]_{r^\prime-1}$, multiplying this certificate by $p(\bx,\bg(\bx))^2$ gives degree at most
\[
 2d(r^\prime-1)+4d=2d(r^\prime+1)\leq2r.
\]
Consequently,
\[
 \ell_{\biy^{2r,2r^\prime}}(u_j(1-u_j)p^2)
 =\ell_{\biy}(g_j(1-g_j)p(\bx,\bg(\bx))^2)\geq0.
\]
The vanishing equations follow from substitution and the full truncated equality ideal. This completes the proof.
\end{proof}

Following the preceding lemmas, we now analyze the convergence rate of the Putinar-type hierarchy. Fix the degree $k$ of the moments to be approximated, and set
\begin{equation}\label{eq: residual constant}
\begin{split}
 t&:=\max\{k,2d,2\deg h_1,\ldots,2\deg h_p\},\\
 C_{\bg,\bh}&:=\sum_{i=1}^p\monc{h_i^2}
              +\sum_{j=1}^m\monc{(g_j-u_j)^2}\\
 &\leq\sum_{i=1}^p\mon{h_i}^{2}
       +\sum_{j=1}^m(1+\mon{g_j})^2.
\end{split}
\end{equation}
Here and below the maximum ignores an empty list. The final inequality follows from the product inequalities in Appendix~\ref{app:norms}.

\begin{theorem}\label{thm: convergence rate Putinar-type}
Let Assumption~\ref{assume: hypercube containment} hold, with $\mon{g_j}\leq1$. Choose an integer $r^\prime\geq\max\{2,t\}$ such that
\[
 d(r^\prime+2)\leq r,\qquad n+m\mid r^\prime,\qquad
 \frac{r^\prime}{\log_2r^\prime}\geq c_2(n+m)^2t^2.
\]
Then
\begin{equation}\label{eq: Putinar general bound}
 \bd_k(\CQ(\cX)_{2r})
 \leq\varepsilon+c_{\Lo}\Lip_k(\mathrm{H}^n)
                    (C_{\bg,\bh}\varepsilon)^{\Lo/2},
\end{equation}
where
\[
 \varepsilon=c_1(n+m)^3t^2\sqrt{s(n+m,t)-1}
                     \frac{\log_2^3r^\prime}{(r^\prime)^2}.
\]
For every $f\in\RR[\bx]_k$, with $f_{\min}=\min_{\bx\in\cX}f(\bx)$, the corresponding optimal-value bound is
\[
 0\leq f_{\min}-\lb(f,\CQ(\cX))_r
 \leq\monc f\left(\varepsilon+c_{\Lo}\Lip_k(\mathrm H^n)
                      (C_{\bg,\bh}\varepsilon)^{\Lo/2}\right).
\]
In particular, for fixed problem data, fixed $k$, and fixed $f\in\RR[\bx]_k$,
\[
\begin{aligned}
 \bd_k(\CQ(\cX)_{2r})
   &=\mathrm{O}\!\left(\frac{(\log_2 r)^{3\Lo/2}}{r^{\Lo}}\right),\\
 f_{\min}-\lb(f,\CQ(\cX))_r
   &=\mathrm{O}\!\left(\frac{(\log_2 r)^{3\Lo/2}}{r^{\Lo}}\right).
\end{aligned}
\]
\end{theorem}
\begin{proof}
Let $\biy$ be an arbitrary element of $\CP\CM_k(\CQ(\cX)_{2r})$. In what follows, we perform the same sequence of lifts and projections as in Section~\ref{sec: method}.
\begin{itemize}
\item Choose a full feasible extension of $\biy$ through degree $2r$. Apply the substitution map to this extension, and denote its degree-$t$ lift by $\biy^{2r,t}$. Lemma~\ref{lem: Putinar pseudo moment sequence for box} places it in $\CP\CM_t(\CQ(\mathrm{H}^{n+m})_{2r^\prime})$.
\item Let $\overline{\biy}^{\prime}$ be its projection onto $\CM_t(\mathrm{H}^{n+m})$, and take an atomic representation
\[
 \overline{\biy}^{\prime}=\sum_{s=1}^Nw_s\bv_t(\bx_s,\bu_s),\qquad
 w_s\geq0,\quad\sum_{s=1}^Nw_s=1,\quad(\bx_s,\bu_s)\in\mathrm{H}^{n+m}.
\]
Here $N\leq s(n+m,t)$; the estimates below do not use this bound on $N$.
\item Define $\overline{\biy}=\sum_sw_s\bv_k(\bx_s)$. For each $s$, let $\widetilde{\bx}_s$ be a closest point in $\cX$ to $\bx_s$, and set
\[
 \widetilde{\biy}=\sum_sw_s\bv_k(\widetilde{\bx}_s),\qquad
 \widetilde{\biy}^{\prime}=\sum_sw_s\bv_t(\widetilde{\bx}_s,\bg(\widetilde{\bx}_s)).
\]
\end{itemize}

We evaluate the two quantities in~\eqref{eq: universal moment transfer}. First, Theorem~\ref{thm: Hausdorff over H^n} gives
\[
 \|\biy-\overline{\biy}\|\leq\|\biy^{2r,t}-\overline{\biy}^{\prime}\|\leq\varepsilon.
\]
Next, we bound the average squared violation of the $\bx$-atoms.
The lifted functional annihilates all squared equations. The Cauchy--Schwarz inequality for coefficient vectors therefore gives
\begin{align}
 \sum_sw_sh_i(\bx_s)^2
 &=\ell_{\overline{\biy}^{\prime}-\biy^{2r,t}}(h_i^2)
 \leq\varepsilon\monc{h_i^2},\label{eq: atom ine 1}\\
 \sum_sw_s(g_j(\bx_s)-u_{s,j})^2
 &\leq\varepsilon\monc{(g_j-u_j)^2}.\label{eq: Put Lo 2}
\end{align}
Since $u_{s,j}\geq0$, we have
$\max\{0,-g_j(\bx_s)\}\leq|g_j(\bx_s)-u_{s,j}|$. Thus
\[
 \sum_sw_s\bp(\bx_s)^2
 \leq\sum_sw_s\left(\sum_i h_i(\bx_s)^2+
                          \sum_j(g_j(\bx_s)-u_{s,j})^2\right)
 \leq C_{\bg,\bh}\varepsilon.
\]
Apply~\eqref{eq: universal moment transfer} with $e_r=\varepsilon$ and $\eta_r=C_{\bg,\bh}\varepsilon$, and take the supremum over $\biy$. This proves~\eqref{eq: Putinar general bound}. For the asymptotic rate, take the largest multiple of $n+m$ not exceeding $r/d-2$; it is proportional to $r$ and satisfies the remaining conditions for all sufficiently large $r$. The optimal-value bound follows from Lemma~\ref{lemma: distance to convergence rate} and strong duality.
\end{proof}

For both bounds, a linear error bound ($\Lo=1$) gives $\mathrm{O}((\log_2 r)^{3/2}/r)$, whereas an error bound with $\Lo=1/2$ gives $\mathrm{O}((\log_2 r)^{3/4}/\sqrt r)$.

\section[Krivine--Stengle-type and Extended-Handelman-type moment-SOS hierarchy]{Krivine--Stengle-type and Extended-Handelman-type\\ moment-SOS hierarchy}\label{sec: LP}

This section aims to study the convergence rate of the Krivine--Stengle-type and the extended-Handelman-type moment-SOS hierarchies. After an affine change of variables, we assume that the nonempty set $\cX$ is contained in $\mathrm{H}^n$. For the sake of simplicity, we include the coordinate inequalities explicitly and normalize $\mon{g_j}\leq1$ for all $j\in[m]$. Thus
\[
 \cX=\{\bx\in\RR^n:x_i\geq0,\;1-x_i\geq0\ (i\in[n]),\;
             g_j(\bx)\geq0\ (j\in[m]),\;h_i(\bx)=0\ (i\in[p])\}.
\]
We consider the extended-Handelman-type hierarchy~\eqref{hierarchy: LP}, where
\begin{multline*}
 \CR(\cX)_r=\Big\{\sum_{\alpha,\beta,\gamma}
      \sigma_{\alpha,\beta,\gamma}\bg^\alpha\bx^\beta(1-\bx)^\gamma
        +\sum_{i=1}^p\tau_i h_i:\;\sigma_{\alpha,\beta,\gamma}\geq0,\;\tau_i\in\RR[\bx],\\
 \deg(\bg^\alpha)+|\beta|+|\gamma|\leq r,\quad
 \deg\tau_i+\deg h_i\leq r\Big\}.
\end{multline*}
The normalized Krivine--Stengle preprime also includes factors $1-g_j$. It contains the displayed cone, so every estimate below applies to it as well. For a description without the explicit coordinate generators, the same transfer is valid once fixed-degree certificates for those generators have been supplied; the associated degree overhead must then be included.

We next follow the two stages described in Section~\ref{sec: method}.
\subsection{Stage 1: Base estimation}
The extended-Handelman-type hierarchy has the dual formulation~\eqref{hierarchy: LP}, whose primal form is
\begin{equation*}
 \mlb(f,\CR(\cX))_r=\min\{\ell_{\biy}(f):\biy\in\CP\CM(\CR(\cX)_r)\},
\end{equation*}
where
\begin{multline*}
 \CP\CM(\CR(\cX)_r)=\{\biy\in\RR^{s(n,r)}:\iy_0=1,\;
  \ell_{\biy}(\bx^\alpha h_i)=0\quad(|\alpha|+\deg h_i\leq r),\\
 \ell_{\biy}(\bg^\alpha\bx^\beta(1-\bx)^\gamma)\geq0
 \quad(\deg(\bg^\alpha)+|\beta|+|\gamma|\leq r)\}.
\end{multline*}
\begin{lemma}
For positive integers $k\leq r$, the set $\CP\CM_k(\CR(\cX)_r)$ is compact and contained in the Euclidean ball centered at the origin of radius $\sqrt{s(n,k)}$.
\end{lemma}
\begin{proof}
We show by induction that
\begin{equation}\label{eq: bound}
 0\leq\ell_{\biy}(\bx^\alpha)\leq1\qquad(|\alpha|\leq r).
\end{equation}
The left-hand inequality follows from the defining positivity constraints. For the right-hand inequality, $1-x_i\in\CR(\mathrm{H}^n)_1$. If $1-\bx^\alpha\in\CR(\mathrm{H}^n)_t$, then
\[
 1-\bx^{\alpha+e_i}=x_i(1-\bx^\alpha)+1-x_i
                  \in\CR(\mathrm{H}^n)_{t+1}.
\]
This proves~\eqref{eq: bound}. The full feasible set is closed and bounded, and its coordinate projection is therefore compact. The radius bound follows by summing the squares of its components.
\end{proof}
The primal feasible set is nonempty and compact. Linear programming duality therefore gives equality of the moment and certificate bounds. We next recall the error bound of de Klerk and Laurent~\cite[Theorem 3.4(i)]{deKlerkLaurent2010} for the Handelman cone on the hypercube. This is the base estimate for the Extended-Handelman-type hierarchy considered here.

\begin{lemma}[\cite{deKlerkLaurent2010}, Theorem 3.4(i)]\label{lem: rate on hypercube Handelman}
Let $f\in\RR[\bx]_k$, with $k\geq1$, and let $r\geq\max\{n,k\}$. Then
\[
 f_{\min}-\mlb(f,\CR(\mathrm{H}^n))_r
 =f_{\min}-\lb(f,\CR(\mathrm{H}^n))_r
 \leq\frac{L(f)}{\lfloor r/n\rfloor}\binom{k+1}{3}n^k,
\]
where $f_{\min}=\min_{\mathrm{H}^n}f$ and
\[
 L(f):=\max_{\alpha\in\NN^n_k}|f_\alpha|\frac{\alpha!}{|\alpha|!},
 \qquad f=\sum_{\alpha\in\NN^n_k}f_\alpha\bx^\alpha.
\]
\end{lemma}
\begin{proof}
Set $\ell=\lfloor r/n\rfloor\geq1$. Theorem 3.4(i) of~\cite{deKlerkLaurent2010}, with polynomial degree $k$ and Bernstein order $\ell$, gives
\[
 f-f_{\min}+\frac{L(f)}\ell\binom{k+1}{3}n^k
 \in\CR(\mathrm{H}^n)_{\max\{n\ell,k\}}.
\]
Since $\max\{n\ell,k\}\leq r$, the claimed bound follows. Thus the integer in the denominator of the cited theorem is the Bernstein order; the certificate degree is at most the maximum of $n$ times that order and the objective degree.
\end{proof}

\begin{theorem}\label{thm: bound on Hausdorff Handelman over hypercube}
For positive integers $k$ and $r\geq\max\{n,k\}$,
\[
 \bd_k(\CR(\mathrm{H}^n)_r)
 \leq\frac{n^k}{\lfloor r/n\rfloor}\binom{k+1}{3}.
\]
For every $f\in\RR[\bx]_k$, with $f_{\min}=\min_{\bx\in\mathrm H^n}f(\bx)$, the corresponding optimal-value bound is
\[
 0\leq f_{\min}-\lb(f,\CR(\mathrm H^n))_r
 \leq\monc f\,\frac{n^k}{\lfloor r/n\rfloor}\binom{k+1}{3}.
\]
In particular, for fixed $n,k$ and $f$, both the Hausdorff distance and the optimal-value gap converge at rate $\mathrm O(1/r)$.
\end{theorem}
\begin{proof}
Choose a pseudo-moment sequence $\overline{\biy}$ attaining the directed Hausdorff distance and let $\widetilde{\biy}$ be its projection onto $\CM_k(\mathrm{H}^n)$. Repeat the first-order optimality argument in Theorem~\ref{thm: Hausdorff over H^n}, with
\[
 f(\bx)=\langle\widetilde{\biy}-\overline{\biy},\bv_k(\bx)\rangle.
\]
Both moment sequences have constant component $1$, so $f$ has zero constant coefficient. Since $\alpha!/|\alpha|!\leq1$,
\[
 L(f)\leq\max_{\alpha\ne0}|\widetilde{\iy}_\alpha-\overline{\iy}_\alpha|
       \leq\|\widetilde{\biy}-\overline{\biy}\|.
\]
Lemma~\ref{lem: rate on hypercube Handelman} therefore gives
\[
 \bd_k(\CR(\mathrm{H}^n)_r)^2
 \leq\frac{n^k}{\lfloor r/n\rfloor}\binom{k+1}{3}
                         \|\widetilde{\biy}-\overline{\biy}\|.
\]
Division by the distance, if it is nonzero, proves the claim. In particular, no radius factor is needed: the auxiliary objective can be used without a constant shift. The optimal-value bound follows from Lemma~\ref{lemma: distance to convergence rate} and LP duality. Since $\lfloor r/n\rfloor$ is proportional to $r$ for fixed $n$, both rates follow.
\end{proof}

\subsection{Stage 2: Lift and projection}
We perform the same sequence of lifts and projections as in Section~\ref{sec: Put lift and project}, reusing $\cK$ and $\varphi_k^t$. The difference is that we now use $\CR(\cK)_r$, generated by the factors $x_i$, $1-x_i$, $u_j$ and $1-u_j$ together with the full ideal of the equations $h_i=0$ and $\widehat{g}_j:=g_j-u_j=0$. We retain $d$, $t$ and $C_{\bg,\bh}$ from~\eqref{eq: residual constant}.
\begin{lemma}\label{lem: Handelman pseudo moment sequence}
Let $r^\prime=\lfloor r/d\rfloor$. For every $\biy\in\CP\CM(\CR(\cX)_r)$ and $t\leq r^\prime$,
\[
 \biy^{r,t}\in\CP\CM_t(\CR(\cK)_{r^\prime})
       \subseteq\CP\CM_t(\CR(\mathrm{H}^{n+m})_{r^\prime}).
\]
\end{lemma}
\begin{proof}
By the induction used to prove~\eqref{eq: bound}, $1-\bx^\alpha\in\CR(\mathrm{H}^n)_{|\alpha|}$. Also $1+\bx^\alpha$ belongs to that cone. Since $\mon{g_j}\leq1$,
\[
 1-g_j=1-\mon{g_j}+\sum_\alpha|g_{j,\alpha}|
                    (1-\operatorname{sign}(g_{j,\alpha})\bx^\alpha)
                 \in\CR(\mathrm{H}^n)_d.
\]
Under substitution, the factors $u_j$ and $1-u_j$ become $g_j$ and $1-g_j$, respectively. The preprime is closed under multiplication, with degrees adding. Thus an allowed product of degree at most $r^\prime$ pulls back to a certificate of degree at most $dr^\prime\leq r$. Normalization is preserved. The equations $g_j-u_j$ vanish identically, and a permitted multiple of $h_i$ pulls back to a permitted multiple of $h_i$. This proves all the required conditions.
\end{proof}

\begin{theorem}\label{thm: convergence rate Handelman}
Let the above assumptions hold, and let $r^\prime=\lfloor r/d\rfloor$ satisfy $r^\prime\geq\max\{n+m,t\}$. Then
\[
 \bd_k(\CR(\cX)_r)
 \leq\varepsilon+c_{\Lo}\Lip_k(\mathrm{H}^n)
                  (C_{\bg,\bh}\varepsilon)^{\Lo/2},
 \qquad
 \varepsilon=\frac{(n+m)^t}{\lfloor r^\prime/(n+m)\rfloor}\binom{t+1}{3}.
\]
For every $f\in\RR[\bx]_k$, with $f_{\min}=\min_{\bx\in\cX}f(\bx)$, the corresponding optimal-value bound is
\[
 0\leq f_{\min}-\lb(f,\CR(\cX))_r
 \leq\monc f\left(\varepsilon+c_{\Lo}\Lip_k(\mathrm H^n)
                      (C_{\bg,\bh}\varepsilon)^{\Lo/2}\right).
\]
In particular, for fixed problem data, fixed $k$, and fixed $f\in\RR[\bx]_k$,
\[
 \bd_k(\CR(\cX)_r)=\mathrm O(1/r^{\Lo/2}),\qquad
 f_{\min}-\lb(f,\CR(\cX))_r=\mathrm O(1/r^{\Lo/2}).
\]
\end{theorem}
\begin{proof}
Let $\biy\in\CP\CM_k(\CR(\cX)_r)$. Choose a full feasible extension and lift it by Lemma~\ref{lem: Handelman pseudo moment sequence}. Project its degree-$t$ lift $\biy^{r,t}$ onto $\CM_t(\mathrm{H}^{n+m})$, obtaining
\[
 \overline{\biy}^{\prime}=\sum_sw_s\bv_t(\bx_s,\bu_s),\qquad
 w_s\geq0,\quad\sum_sw_s=1.
\]
As before, let $\widetilde{\bx}_s$ be a closest point to $\bx_s$ in $\cX$ and set
\[
 \overline{\biy}=\sum_sw_s\bv_k(\bx_s),\qquad
 \widetilde{\biy}=\sum_sw_s\bv_k(\widetilde{\bx}_s),\qquad
 \widetilde{\biy}^{\prime}=\sum_sw_s\bv_t(\widetilde{\bx}_s,\bg(\widetilde{\bx}_s)).
\]

Theorem~\ref{thm: bound on Hausdorff Handelman over hypercube} gives
\[
 \|\biy-\overline{\biy}\|\leq\|\biy^{r,t}-\overline{\biy}^{\prime}\|\leq\varepsilon.
\]
The lifted functional annihilates $h_i^2$ and $(g_j-u_j)^2$. Applying the coefficient Cauchy--Schwarz estimate as in~\eqref{eq: atom ine 1}--\eqref{eq: Put Lo 2}, and using $u_{s,j}\geq0$, gives
\[
 \sum_sw_s\bp(\bx_s)^2
 \leq\sum_sw_s\left(\sum_i h_i(\bx_s)^2+
                    \sum_j(g_j(\bx_s)-u_{s,j})^2\right)
 \leq C_{\bg,\bh}\varepsilon.
\]
The common estimate~\eqref{eq: universal moment transfer}, with $e_r=\varepsilon$ and $\eta_r=C_{\bg,\bh}\varepsilon$, proves the bound. Since $r^\prime$ is proportional to $r$ for fixed $d$, the asserted rate follows. The optimal-value bound follows from Lemma~\ref{lemma: distance to convergence rate} and LP duality.
\end{proof}
The exponent $\Lo/2$ agrees with the effective-degree result in~\cite{heijmans2026degree}; here it follows from a uniform approximation of the truncated pseudo-moment set.

\section{Hol-Scherer-type moment-SOS hierarchy}\label{sec:matrix}
In this section, we study the hierarchy~\eqref{hierarchy: Hol and Scherer} for a polynomial matrix inequality. Throughout the section, $n$ denotes the dimension of $\bx$, $m$ denotes the matrix size, $d:=\max\{1,\deg G\}$, and $k\geq1$ is the fixed degree of the moments to be approximated.

\begin{assume}\label{assume: matrix cube}
The symmetric polynomial matrix $G\in\mathbb S^m[\bx]$ defines a nonempty set
\[
 \cX=\{\bx\in[-1,1]^n:G(\bx)\succeq0\}.
\]
The coordinate inequalities $1-x_i^2\geq0$, $i\in[n]$, are included explicitly among the generators of the hierarchy.
\end{assume}
Thus, in this section, the truncated matrix quadratic module is
\begin{equation}\label{eq: matrix module}
 \CQ(\cX)_{2r}=\left\{\sigma_0+\sum_{i=1}^n\sigma_i (1-x_i^2)
                  +\langle R,G\rangle\right\},
\end{equation}
where $\sigma_0,\sigma_i$ are SOS polynomials, $R$ is an SOS matrix, and
$\deg\sigma_0\leq2r$, $\deg\sigma_i\leq2r-2$, and
$\deg R\leq2(r-\lceil G\rceil)$. Here $\lceil G\rceil=\lceil\deg G/2\rceil$ and $\langle R,G\rangle=\tr(RG)$.
A full feasible pseudo-moment sequence $\biy=(\iy_\alpha)_{|\alpha|\leq2r}$ is characterized by its Riesz functional $\ell_{\biy}:\RR[\bx]_{2r}\to\RR$ satisfying
\[
 \ell_{\biy}(1)=1,\qquad \ell_{\biy}(q)\geq0\quad\text{for every }q\in\CQ(\cX)_{2r}.
\]
Here $\CQ([-1,1]^n)$ is generated by $1-x_i^2$, $i\in[n]$. The monomial identities used to prove~\eqref{eq: coordinate certificate}, now with these generators, give
\[
 1\pm\bx^\alpha\in\CQ([-1,1]^n)_{2\lceil|\alpha|/2\rceil}.
\]
Consequently, $|\iy_\alpha|\leq1$ for $|\alpha|\leq2r$, and the normalized pseudo-moment set is compact. The same coefficient-neighborhood argument as in Section~\ref{sec: Put} makes $1$ an interior point of the truncated cone and gives equality of the primal and dual lower bounds. The coordinate constraints in~\eqref{eq: matrix module} are part of feasibility throughout the proof.

Set $h:=\max\{k,2d\}$. We use the constants
\begin{equation}\label{eq: matrix constants}
 \begin{split}
 S_G&:=\sum_{a,b=1}^m\cheb{g_{ab}}^{\,2},\\
 c_k&:=\max_{\substack{\|\bu\|=1\\u_0=0}}
       \cheb{\langle\bu,\bv_k\rangle}
       \leq\sqrt{s(n,k)-1}.
 \end{split}
\end{equation}
The bound for $c_k$ follows from Proposition~\ref{prop:Cheb-norms}, since $u_0=0$. If $\gamma:=\max_{a,b}\cheb{g_{ab}}$, then $S_G\leq m^2\gamma^2$ and $\sup_{[-1,1]^n}\|G\|_{\op}\leq m\gamma$, as proved in Appendix~\ref{app: composition}. Choose any fixed $\rho>0$ with $\rho\geq\sup_{[-1,1]^n}\|G\|_{\op}$, and define
\begin{equation*}
 \cD:=\{(\bx,U):\bx\in[-1,1]^n,\;0\preceq U\preceq\rho I_m\},
 \qquad H(\bx,U):=\|G(\bx)-U\|_F^2.
\end{equation*}
We next define
\begin{equation*}
 c_f:=\min_{(\bx,U)\in\cD}\{f(\bx)+H(\bx,U)\},
 \qquad f\in\RR[\bx]_k.
\end{equation*}
Write $G=G_+-G_-$ for the positive and negative spectral parts, so that $G_+,G_-\succeq0$ and $\langle G_+,G_-\rangle=0$. For $0\preceq U\preceq\rho I_m$,
\[
 \|G(\bx)-U\|_F^2
 =\|G_+(\bx)-U\|_F^2+\|G_-(\bx)\|_F^2
   +2\langle G_-(\bx),U\rangle
 \geq\|G_-(\bx)\|_F^2.
\]
Equality holds at $U=G_+(\bx)$, which satisfies $0\preceq G_+(\bx)\preceq\rho I_m$. Consequently,
\begin{equation}\label{eq: matrix cf spectral}
 c_f=\min_{\bx\in[-1,1]^n}\{f(\bx)+\|G_-(\bx)\|_F^2\},
 \qquad |c_f|\leq\supnorm f{[-1,1]^n}\leq\cheb{f}.
\end{equation}
For the bound on $c_f$, use $H\geq0$ for the lower bound and evaluate at a point of the nonempty set $\cX$ for the upper bound.

\paragraph{The kernel.}
We recall the SOS kernel estimate of Gribling et al.~\cite[Theorem~3 and its proof]{gribling2026squared}. We use input degree $h=\max\{k,2d\}$, which covers $f$, every entry $g_{ab}$, and every square $g_{ab}^2$. Let
\[
 d\mu(\bx):=\prod_{i=1}^n\frac{dx_i}{\pi\sqrt{1-x_i^2}}
 \quad\text{on }[-1,1]^n
\]
be the fixed Chebyshev probability measure. We reserve degree $2d$ for the squared residual, so the degree available for the kernel is $2r-2d$. The next lemma is modified from~\cite[Theorem 3 and its proof]{gribling2026squared}.

\begin{lemma}\label{lem: matrix kernel}
There exist absolute constants $0<c_1\leq70{,}458$ and $1\leq c_2\leq540$ such that, for every integer $r>d$ satisfying
\[
 n\mid 2(r-d),\qquad
 \frac{2(r-d)}{\log_2(2(r-d))}\geq c_2n^2h^2,
\]
there is a polynomial kernel $K_{2r}(\bt,\bx)$ that is SOS in $\bt$ for every $\bx\in[-1,1]^n$, with $\deg_{\bt}K_{2r}\leq2r-2d$. Its operator
\begin{equation}\label{eq: matrix kernel}
 (\mathcal T_{2r} p)(\bt):=\int_{[-1,1]^n}K_{2r}(\bt,\bx)p(\bx)\,d\mu(\bx)
\end{equation}
satisfies, for every $p\in\RR[\bx]_h$,
\begin{equation}\label{eq: matrix kernel approximation}
 \cheb{E_p}\leq\varepsilon_r\cheb p,
 \qquad \varepsilon_r:=c_1n^3h^2\frac{(\log_2(r-d))^3}{(r-d)^2},\quad  
\end{equation}
where 
\begin{equation}\label{eq: matrix kernel error}
    E_p:=\mathcal T_{2r} p-p \quad \forall p \in \RR[\bx].
\end{equation}
\end{lemma}

\subsection{Stage 1: Base estimation}
In this stage, we establish a bound on $c_f-\ell_{\biy}(f)$ for every full feasible pseudo-moment sequence $\biy$. The reason we focus on the function 
\begin{equation*}
 c_f:=\min_{(\bx,U)\in\cD}\{f(\bx)+H(\bx,U)\},
 \qquad f\in\RR[\bx]_k.
\end{equation*}
is for any $\bx \in \cX$, $\min_{U \succeq 0} H(\bx,U) = 0$, and for any $\biy \in \CP\CM_k(\CQ(\CX)_{2r})$, $\biy$ should be closed to $\CM_k(\CX)$, which intuitively shows 
\begin{displaymath}
    \ell_{\biy}(c_f - f) = c_f - \ell_{\biy}(f) \geq 0.
\end{displaymath}
We first record the norm bound used to evaluate kernel errors.

\begin{lemma}\label{lem: matrix functional norm}
Let $\ell_{\biy}$ be normalized and nonnegative on $\CQ([-1,1]^n)_{2r}$. Then, for every $p\in\RR[\bx]_{2r}$,
\begin{equation*}
 |\ell_{\biy}(p)|\leq\cheb{p}.
\end{equation*}
\end{lemma}
\begin{proof}
We prove that $1\pm T_\alpha\in\CQ([-1,1]^n)_{2r}$ for $|\alpha|\leq2r$, where $T_\alpha(\bx)=\prod_iT_{\alpha_i}(x_i)$. Distribute the degrees $\alpha_i$ into two groups of degree at most $r$, splitting at most one coordinate $j$. This gives
\[
 T_\alpha=P Q T_{a+b}(x_j),\qquad
 \deg P+a\leq r,\qquad \deg Q+b\leq r,
\]
where $a,b\geq0$, and $P,Q$ are products of the remaining Chebyshev factors on disjoint sets of coordinates. Empty products equal one. Let $U_i$ denote the Chebyshev polynomial of the second kind, with $U_{-1}=0$. The identities
\[
 T_{a+b}=T_aT_b-(1-t^2)U_{a-1}U_{b-1},\qquad
 T_a^2+(1-t^2)U_{a-1}^2=1
\]
give the certificate
\[
\begin{aligned}
 1\pm T_\alpha=\frac12\big[&
 (P T_a(x_j)\pm Q T_b(x_j))^2\\
 &+(1-x_j^2)(P U_{a-1}(x_j)\mp Q U_{b-1}(x_j))^2\\
 &+(1-P^2)+(1-Q^2)\big].
\end{aligned}
\]
The first square has degree at most $2r$, and the polynomial squared in the localizing term has degree at most $r-1$. Moreover, if $P=\prod_{i\in I}T_{\alpha_i}(x_i)$, then
\[
 1-P^2=\sum_{\substack{i\in I\\\alpha_i>0}}(1-x_i^2)
 \left(U_{\alpha_i-1}(x_i)
       \prod_{\substack{a\in I\\a<i}}T_{\alpha_a}(x_a)\right)^2
 \in\CQ([-1,1]^n)_{2\deg P}.
\]
The same identity applies to $Q$, and both degrees are at most $r$. Hence the certificate belongs to $\CQ([-1,1]^n)_{2r}$. Applying $\ell_{\biy}$ gives $|\ell_{\biy}(T_\alpha)|\leq1$. Expanding $p$ in the Chebyshev basis and using the triangle inequality proves the result.
\end{proof}

\begin{lemma}\label{lem: matrix kernel certificate}
Let $f\in\RR[\bx]_k$ and let $K_{2r}(\bt,\bx)$ be a polynomial kernel that is SOS in $\bt$ for every $\bx\in[-1,1]^n$, with $\deg_{\bt}K_{2r}+2d\leq2r$. Then the polynomial
\begin{equation*}
 R_f(\bt):=\int_{[-1,1]^n}K_{2r}(\bt,\bx)
       \big(f(\bx)+\|G(\bx)-G(\bt)\|_F^2-c_f\big)\,d\mu(\bx)
\end{equation*}
belongs to $\CQ(\cX)_{2r}$. In particular, this holds for the kernel in~\eqref{eq: matrix kernel}.
\end{lemma}
\begin{proof}
For a fixed $\bx\in[-1,1]^n$, set
$a_f(\bx):=f(\bx)+\|G_-(\bx)\|_F^2-c_f\geq0$.
The orthogonality of $G_+(\bx)$ and $G_-(\bx)$ gives
\[
 \|G(\bx)-G(\bt)\|_F^2
 =\|G(\bt)-G_+(\bx)\|_F^2
   +2\langle G_-(\bx),G(\bt)\rangle+\|G_-(\bx)\|_F^2.
\]
Therefore
\begin{equation*}
 R_f(\bt)=\sigma_f(\bt)+\langle R(\bt),G(\bt)\rangle,
\end{equation*}
where
\begin{align*}
 \sigma_f(\bt)&=\int_{[-1,1]^n}K_{2r}(\bt,\bx)
          \big(\|G(\bt)-G_+(\bx)\|_F^2+a_f(\bx)\big)\,d\mu(\bx),\\
 R(\bt)&=2\int_{[-1,1]^n}K_{2r}(\bt,\bx)G_-(\bx)\,d\mu(\bx).
\end{align*}
For each fixed $\bx$, the first integrand is SOS in $\bt$, and the second is an SOS matrix in $\bt$. Indeed, in a fixed monomial basis, write $K_{2r}(\bt,\bx)=\bv_a(\bt)^\top Q(\bx)\bv_a(\bt)$ with $2a=\deg_{\bt}K_{2r}$ and $Q(\bx)\succeq0$. For any bounded nonnegative weight $b(\bx)$,
\[
 \int K_{2r}(\bt,\bx)b(\bx)\,d\mu(\bx)
 =\bv_a(\bt)^\top\left(\int b(\bx)Q(\bx)\,d\mu(\bx)\right)\bv_a(\bt).
\]
The integrated Gram matrix is positive semidefinite, so this polynomial is SOS and its degree in $\bt$ does not increase. The weight need not be polynomial: it only changes the coefficients in $\bt$. The same Gram-matrix argument applies to the displayed scalar and matrix integrands. Hence $\sigma_f$ is SOS and $R$ is an SOS matrix. The degree bounds required by~\eqref{eq: matrix module} are
\[
 \deg\sigma_f\leq\deg_{\bt}K_{2r}+2d\leq2r,\qquad
 \deg R\leq\deg_{\bt}K_{2r}\leq2r-2d\leq2(r-\lceil G\rceil).
\]
Thus $R_f\in\CQ(\cX)_{2r}$. In particular, the bound on $R$ ensures that the product $\langle R,G\rangle$ has degree at most $2r$.

The coefficient integrals are finite because $G_+$, $G_-$, and $a_f$ are continuous on the compact integration domain. The functional $\ell_{\biy}$ is applied only to the resulting polynomials in $\bt$.
\end{proof}

\begin{proposition}[Base estimate]\label{prop: matrix base estimate}
Let $r$ satisfy the conditions of Lemma~\ref{lem: matrix kernel}. For every $f\in\RR[\bx]_k$ and every full feasible pseudo-moment sequence $\biy$ for~\eqref{eq: matrix module},
\begin{equation}\label{eq: matrix base estimate}
 c_f-\ell_{\biy}(f)\leq2\varepsilon_r\cheb{f}+4S_G\varepsilon_r.
\end{equation}
\end{proposition}
\begin{proof}
Define the SOS polynomial
\[
 V_G(\bt):=\int_{[-1,1]^n}K_{2r}(\bt,\bx)
                         \|G(\bt)-G(\bx)\|_F^2\,d\mu(\bx).
\]
For a single entry $g=g_{ij}$, its contribution to $V_G$ is
\[
 g^2\mathcal T_{2r}1-2g\mathcal T_{2r} g+\mathcal T_{2r}(g^2)
 =g^2E_1-2gE_g+E_{g^2}.
\]
Here $E_g=\mathcal T_{2r} g-g$ and $E_{g^2}=\mathcal T_{2r}(g^2)-g^2$, as in~\eqref{eq: matrix kernel error}; in particular, $E_{g^2}$ is not the square of $E_g$. The terms $g^2-2g^2+g^2$ cancel. The term $E_1$ is retained because the polynomial normalization need not give $\mathcal T_{2r}1=1$ exactly.

Since $\deg g^2\leq2d\leq h$, the kernel estimate~\eqref{eq: matrix kernel approximation} and the product inequality~\eqref{eq:Cheb-product} give
\[
\begin{aligned}
 \cheb{E_1}&\leq\varepsilon_r,\\
 \cheb{E_g}&\leq\varepsilon_r\cheb{g},\\
 \cheb{E_{g^2}}&\leq\varepsilon_r\cheb{g}^{\,2}.
\end{aligned}
\]
Summing the contributions of the entries yields
\begin{equation}\label{eq: matrix variance bound}
 \cheb{V_G}\leq4S_G\varepsilon_r,
 \qquad 0\leq \ell_{\biy}(V_G)\leq4S_G\varepsilon_r.
\end{equation}
The lower bound follows from $V_G\in\Sigma_{2r}[\bt]$. For the upper bound, apply Lemma~\ref{lem: matrix functional norm}, using $\deg V_G\leq\deg_{\bt}K_{2r}+2d\leq2r$ from Lemma~\ref{lem: matrix kernel}.

Lemma~\ref{lem: matrix kernel certificate} now gives
\[
 0\leq \ell_{\biy}(R_f)=\ell_{\biy}(f)+\ell_{\biy}(E_f)+\ell_{\biy}(V_G)-c_f\big(1+\ell_{\biy}(E_1)\big).
\]
The same functional norm bound,~\eqref{eq: matrix cf spectral}, and~\eqref{eq: matrix variance bound} imply
\begin{align*}
 c_f-\ell_{\biy}(f)
 &\leq \ell_{\biy}(E_f)+\ell_{\biy}(V_G)-c_f \ell_{\biy}(E_1)\\
 &\leq\varepsilon_r\cheb{f}
       +4S_G\varepsilon_r+\varepsilon_r|c_f|\\
 &\leq2\varepsilon_r\cheb{f}+4S_G\varepsilon_r.
\end{align*}
The positivity statements use certificates in $\CQ(\cX)_{2r}$. For the norm estimates, all the conditions on degree are satisfied, i.e., $\deg E_f\leq\max\{\deg_{\bt}K_{2r},k\}\leq2r$, $\deg E_1\leq\deg_{\bt}K_{2r}\leq2r-2d$, and $\deg V_G\leq2r$.
\end{proof}

\subsection{Stage 2: Lift and projection}
We use the base estimate to control the optimal value of a moment projection with the fixed squared residual $H$. To define its Riesz functional precisely, regard $U$ as a symmetric matrix with independent entries indexed by
$\mathcal I:=\{(a,b):1\leq a\leq b\leq m\}$, and write
$U^\beta:=\prod_{(a,b)\in\mathcal I}U_{ab}^{\beta_{ab}}$.
A sequence $\biw=(w_{\alpha,\beta})_{|\alpha|+|\beta|\leq h}$ belongs to $\CM_h(\cD)$ if it is represented by a probability measure on $\cD$. For a polynomial
$P(\bx,U)=\sum_{\alpha,\beta}P_{\alpha,\beta}\bx^\alpha U^\beta$ of total degree at most $h$, define
\begin{equation*}
 \ell_{\biw}(P):=\sum_{|\alpha|+|\beta|\leq h}
                          P_{\alpha,\beta}w_{\alpha,\beta}.
\end{equation*}
In particular, if $g_{ab}(\bx)=\sum_\alpha g_{ab,\alpha}\bx^\alpha$, then
\begin{equation}\label{eq: matrix Riesz H}
 \ell_{\biw}(H)=\sum_{a,b=1}^m\left(
  \sum_{\alpha,\delta}g_{ab,\alpha}g_{ab,\delta}w_{\alpha+\delta,0}
 -2\sum_\alpha g_{ab,\alpha}w_{\alpha,e^{ab}}
 +w_{0,2e^{ab}}\right).
\end{equation}
Here $e^{ab}$ denotes the unit multi-index of the independent entry $U_{\min\{a,b\},\max\{a,b\}}$. Thus the off-diagonal entries are counted twice, as required by the Frobenius norm. Formula~\eqref{eq: matrix Riesz H} is well-defined because $\deg H\leq2d\leq h$. It is the expectation of $\|G(\bx)-U\|_F^2$ under a representing measure of $\biw$.

Write $\pi_k^x\biw=(w_{\alpha,0})_{|\alpha|\leq k}$. Given $\biy\in\pMkQX2r$, consider
\begin{equation}\label{eq: matrix moment projection}
 \min_{\substack{\biw\in\CM_h(\cD)\\\biz=\pi_k^x\biw}}
          \left\{\|\biy-\biz\|^2+\ell_{\biw}(H)\right\}.
\end{equation}
Since $\cD$ is compact, $\CM_h(\cD)$ is compact and convex, so an optimal pair $(\biz^*,\biw^*)$ exists.

\begin{proposition}\label{prop: matrix projection estimate}
Under the conditions of Proposition~\ref{prop: matrix base estimate}, every optimal pair in~\eqref{eq: matrix moment projection} satisfies
\begin{equation}\label{eq: matrix projection estimate}
 \|\biy-\biz^*\|^2+\ell_{\biw^*}(H)\leq E_r,
 \qquad E_r:=4c_k^2\varepsilon_r^2+4S_G\varepsilon_r.
\end{equation}
\end{proposition}
\begin{proof}
Set $\delta:=\biz^*-\biy$, $e:=\|\delta\|$, and
$f_\delta(\bx):=2\langle\delta,\bv_k(\bx)\rangle$.
The first-order optimality condition for~\eqref{eq: matrix moment projection} is
\[
 2\langle\delta,\pi_k^x\biw-\biz^*\rangle
       +\ell_{\biw}(H)-\ell_{\biw^*}(H)\geq0
       \qquad\forall\biw\in\CM_h(\cD).
\]
Hence $\biw^*$ minimizes the linear functional $\ell_{\biw}(f_\delta+H)$. Its minimum over true moment sequences equals the minimum over $\cD$, because the value is at least that pointwise minimum and equality is attained by a Dirac measure at a minimizer. Therefore
\[
 c_{f_\delta}=2\langle\delta,\biz^*\rangle+\ell_{\biw^*}(H).
\]
Choose a full feasible extension $\biy^{2r}$ of $\biy$. Since $f_\delta$ has degree at most $k$, we have $\ell_{\biy^{2r}}(f_\delta)=\ell_{\biy}(f_\delta)=2\langle\delta,\biy\rangle$. Together with $\delta_0=0$, this gives
\[
 c_{f_\delta}-\ell_{\biy}(f_\delta)=2e^2+\ell_{\biw^*}(H),
 \qquad \cheb{f_\delta}\leq2c_ke.
\]
Applying Proposition~\ref{prop: matrix base estimate} to $\biy^{2r}$ gives
\[
 2e^2+\ell_{\biw^*}(H)\leq4c_k\varepsilon_r e+4S_G\varepsilon_r.
\]
Finally, $4c_k\varepsilon_r e\leq e^2+4c_k^2\varepsilon_r^2$ proves~\eqref{eq: matrix projection estimate}.
\end{proof}

Define the spectral violation function
\[
 \bp(\bx):=\max\{0,-\lambda_{\min}(G(\bx))\}.
\]
It is continuous and semialgebraic, and vanishes exactly on $\cX$ within $[-1,1]^n$. The {\L}ojasiewicz inequality gives constants $c_{\Lo}>0$ and $0<\Lo\leq1$ such that
\begin{equation*}
 \bd(\bx,\cX)\leq c_{\Lo}\bp(\bx)^{\Lo}
                   \qquad\forall\bx\in[-1,1]^n.
\end{equation*}
We retain the same convention for $\Lo$ as in the scalar sections.

\begin{theorem}[Hol-Scherer-type convergence rate]\label{thm:matrix}
Under Assumption~\ref{assume: matrix cube}, let $k\geq1$ be fixed and let $r$ satisfy the conditions in Lemma~\ref{lem: matrix kernel}. With $E_r$ defined in~\eqref{eq: matrix projection estimate},
\begin{equation}\label{eq: matrix Hausdorff}
 \bd_k(\CQ(\cX)_{2r})
 \leq\sqrt{E_r}+c_{\Lo}\Lip_k([-1,1]^n)E_r^{\Lo/2}
 =\mathrm O\left(\frac{(\log_2 r)^{3\Lo/2}}{r^{\Lo}}\right).
\end{equation}
For every $f\in\RR[\bx]_k$, the corresponding optimal-value bound is
\begin{equation}\label{eq: matrix objective rate}
 0\leq f_{\min}-\lb(f,\CQ(\cX))_r
 \leq\monc f\left(\sqrt{E_r}
                    +c_{\Lo}\Lip_k([-1,1]^n)E_r^{\Lo/2}\right).
\end{equation}
The constants and the threshold on $r$ may depend on $n,m,k,G$ and the error-bound constants. The exponent of $r$ has no further dependence on $n$ or $m$ beyond their possible influence on $\Lo$.
\end{theorem}
\begin{proof}
Fix $\biy\in\pMkQX2r$ and choose an optimal pair in~\eqref{eq: matrix moment projection}. An atomic representation of $\biw^*$ on $\cD$ gives
\[
 \biz^*=\sum_sw_s\bv_k(\bx_s),\qquad
 \ell_{\biw^*}(H)=\sum_sw_s\|G(\bx_s)-U_s\|_F^2,
 \qquad w_s\geq0,\quad\sum_sw_s=1.
\]
For any $U\succeq0$,
\[
 \bp(\bx)\leq\|G(\bx)-U\|_{\op}\leq\|G(\bx)-U\|_F.
\]
Indeed, if the smallest eigenvalue of $G(\bx)$ is negative, evaluate $U-G(\bx)$ at a corresponding unit eigenvector and use $U\succeq0$; otherwise the inequality is immediate. Proposition~\ref{prop: matrix projection estimate} therefore yields the two estimates needed for the common atom argument:
\[
 \|\biy-\biz^*\|\leq\sqrt{E_r},\qquad
 \sum_sw_s\bp(\bx_s)^2\leq\ell_{\biw^*}(H)\leq E_r.
\]
Apply~\eqref{eq: universal moment transfer} on $\CY=[-1,1]^n$ with $e_r=\sqrt{E_r}$ and $\eta_r=E_r$, and take the supremum over $\biy$. Since $\CM_k(\cX)\subseteq\pMkQX2r$, this gives~\eqref{eq: matrix Hausdorff}. For fixed $d$, we have
\[
 \frac{(\log_2(r-d))^3}{(r-d)^2}
 \sim\frac{(\log_2 r)^3}{r^2}\qquad(r\to\infty).
\]
The rate follows from~\eqref{eq: matrix kernel approximation} and $0<\Lo\leq1$. It extends to all sufficiently large orders by monotonicity, using $r'=d+n\lfloor(r-d)/n\rfloor\leq r$, for which $n\mid2(r'-d)$ and $r-r'<n$. Finally, Lemma~\ref{lemma: distance to convergence rate} and the strong duality discussed after~\eqref{eq: matrix module} imply~\eqref{eq: matrix objective rate}.
\end{proof}

\section{Conclusion}
In this paper, we have extended the method developed in~\cite{tran2025convergence} from the Schm\"udgen-type hierarchy to the Putinar-type, Krivine--Stengle-type, Extended-Handelman-type, and Hol-Scherer-type hierarchies. The method studies the Hausdorff distance between the sets of truncated pseudo-moment sequences and moment sequences. It consists of two stages: base estimation on a simple set, followed by lift and projection onto the original feasible set. Equality constraints are included in the lifted description, and the same weighted estimate on atoms completes the argument in each case.

For scalar polynomial constraints, we obtain the rates $\mathrm{O}((\log_2 r)^{3\Lo/2}/r^{\Lo})$ for the Putinar-type hierarchy and $\mathrm{O}(r^{-\Lo/2})$ for the normalized Krivine--Stengle-type and Extended-Handelman-type hierarchies. For the Hol-Scherer-type hierarchy, a finite-degree kernel certificate for the squared Frobenius residual gives $\mathrm{O}((\log_2 r)^{3\Lo/2}/r^{\Lo})$ as well. The penalty coefficient is fixed at one. In each case, $\Lo$ is the exponent in the error bound $\bd(\bx,\cX)\leq c_{\Lo}\bp(\bx)^{\Lo}$, and the coordinate constraints are included explicitly: those for $[0,1]^n$ in the scalar cases and for $[-1,1]^n$ in the matrix case. The estimates are uniform for a fixed moment degree and imply convergence bounds for all polynomial objectives of that degree.

Together with our preceding work, these results provide a common method for characterizing convergence rates across different types of moment-SOS hierarchies. The final projection uses the probability weights of a representing measure, so its estimate does not depend on the number of atoms. In the matrix case, the kernel is applied to fixed-degree polynomials; no coefficient estimate for an increasing-degree composition with $G$ is required. Further improvements may come from sharper base estimates, smaller degree overheads in the lifting certificates, and additional structure in the constraints. It would also be interesting to adapt the construction to sparse descriptions.

\appendix
\section{Coefficient norms and Chebyshev norms}\label{app:norms}
This appendix proves the norm inequalities used in the scalar residual estimates, the Putinar-type bound, and the kernel argument for the Hol-Scherer-type hierarchy. The notation $\mon{\cdot}$ and $\monc{\cdot}$ refers to norms of monomial coefficient vectors, whereas $\cheb{\cdot}$ refers to the Chebyshev coefficient norm.

\subsection{The monomial basis and products}
\begin{proposition}\label{prop:mon-norms}
For $p\in\RR[\bx]_d$,
\begin{equation}\label{eq:mon-norms}
 \monc p\leq\mon p\leq\sqrt{s(n,d)}\monc p.
\end{equation}
If at most $M$ coefficients are nonzero, $s(n,d)$ may be replaced by $M$. In particular, if the constant coefficient vanishes, it may be replaced by $s(n,d)-1$. For any polynomials $p,q$,
\begin{equation*}
 \mon{pq}\leq\mon p\mon q,\qquad
 \monc{p^2}\leq\mon p^2.
\end{equation*}
\end{proposition}
\begin{proof}
Since $\sum_\alpha p_\alpha^2\leq(\sum_\alpha|p_\alpha|)^2$, the first inequality in~\eqref{eq:mon-norms} holds. If $p$ has $M$ nonzero coefficients, Cauchy--Schwarz gives
\[
 \sum_{\alpha:p_\alpha\ne0}|p_\alpha|
 \leq\sqrt M\left(\sum_\alpha p_\alpha^2\right)^{1/2}.
\]
For products, $(pq)_\gamma=\sum_{\alpha+\beta=\gamma}p_\alpha q_\beta$, so
\[
 \mon{pq}\leq\sum_\gamma\sum_{\alpha+\beta=\gamma}|p_\alpha||q_\beta|
 =\mon p\mon q.
\]
Finally, $\monc{p^2}\leq\mon{p^2}\leq\mon p^2$.
\end{proof}
The square estimate applies to $h_i$ and $g_j-u_j$ in~\eqref{eq: residual constant}. Since $u_j$ is an additional variable, $\mon{g_j-u_j}=\mon{g_j}+1$.

\subsection{Chebyshev coefficients on \texorpdfstring{$[-1,1]^n$ and $\mathrm{H}^n$}{[-1,1]n and Hn}}
Write
\[
 p(\bx)=\sum_\alpha\widehat p_\alpha T_\alpha(\bx)
        =\sum_\alpha\widetilde p_\alpha T_\alpha(2\bx-\mathbf1).
\]
By definition, $\cheb p=\sum|\widehat p_\alpha|$ and
$\cheb{p\circ\mathbf A}=\sum|\widetilde p_\alpha|$. Both expansions preserve total degree.

\begin{proposition}\label{prop:Cheb-norms}
For $p\in\RR[\bx]_d$,
\begin{align}
 \supnorm p{[-1,1]^n}&\leq\cheb p\leq\mon p\leq\sqrt{s(n,d)}\monc p,\label{eq:Cheb-unshifted}\\
 \cheb{p\circ\mathbf A}&\leq\mon p\leq\sqrt{s(n,d)}\monc p.\nonumber
\end{align}
In both comparisons, the last factor is $\sqrt{s(n,d)-1}$ if the constant coefficient of $p$ vanishes.
Moreover,
\begin{equation}\label{eq:Cheb-product}
 \cheb{pq}\leq\cheb p\cheb q.
\end{equation}
\end{proposition}
\begin{proof}
The uniform estimate follows from $|T_j(t)|\leq1$ for $t\in[-1,1]$. To prove the coefficient comparison, note that
\[
 (\cos t)^a=2^{-a}\sum_{j=0}^a\binom aj\cos((a-2j)t).
\]
Combining equal frequencies expresses $x^a$ as a nonnegative linear combination of Chebyshev polynomials. The coefficients sum to $1$, as is seen by evaluation at $x=1$. Taking products over the coordinates gives $\cheb{\bx^\alpha}=1$, and hence $\cheb p\leq\sum_\alpha|p_\alpha|=\mon p$.

For the shifted basis, the polynomial $((1+x)/2)^a$ has nonnegative monomial coefficients summing to $1$. Apply the preceding expansion to each monomial. Its Chebyshev coefficients are again nonnegative and sum to $1$. Taking products over the coordinates proves $\cheb{\bx^\alpha\circ\mathbf A}=1$. Hence the triangle inequality gives $\cheb{p\circ\mathbf A}\leq\sum_\alpha|p_\alpha|=\mon p$. Proposition~\ref{prop:mon-norms} gives the last inequality and the refinement for a vanishing constant coefficient.

Finally, $T_aT_b=(T_{a+b}+T_{|a-b|})/2$. In several variables, each product $T_\alpha T_\beta$ is a nonnegative linear combination of basis elements whose coefficients sum to $1$. Expand $pq$, apply the triangle inequality, and sum.
\end{proof}

\subsection{Matrix entry norms}\label{app: composition}
For an $m\times m$ polynomial matrix $F$, define
\begin{equation*}
 \|F\|_{1,\mathrm{Cheb},\max}:=\max_{1\leq a,b\leq m}\cheb{F_{ab}}.
\end{equation*}
For the matrix polynomial in Section~\ref{sec:matrix}, set $\gamma=\|G\|_{1,\mathrm{Cheb},\max}$. The bounds used in~\eqref{eq: matrix constants} are
\[
 \sup_{\bx\in[-1,1]^n}\|G(\bx)\|_{\op}\leq m\gamma,
 \qquad S_G:=\sum_{a,b=1}^m\cheb{G_{ab}}^{\,2}\leq m^2\gamma^2.
\]
Indeed,~\eqref{eq:Cheb-unshifted} gives $|G_{ab}(\bx)|\leq\gamma$ on $[-1,1]^n$, and therefore
\[
 \|G(\bx)\|_{\op}\leq\|G(\bx)\|_F
 =\left(\sum_{a,b=1}^m|G_{ab}(\bx)|^2\right)^{1/2}\leq m\gamma.
\]
The estimate for $S_G$ follows by bounding each of its $m^2$ terms by $\gamma^2$.

Multiplication of $G$ by any fixed positive constant preserves its feasible set and truncated matrix quadratic module: a matrix SOS multiplier is rescaled by the reciprocal constant. Thus $\gamma$ may be made smaller than any prescribed positive number. The spectral residual scales by the same factor, so the corresponding {\L}ojasiewicz constant must be rescaled as well. Section~\ref{sec:matrix} does not require such a normalization; its coefficient estimates involve only the entries of $G$ and their squares.

\section*{Conflict Of Interest statement}
The authors have no competing interests to declare that are relevant to the content of this article.
\bibliographystyle{alpha}
\bibliography{main}
\end{document}